%% file: bkk-strong-ima-bib.tex
\documentclass[reqno]{amsart}
\usepackage{amsaddr}
\usepackage{natbib}

\input{bkk_commands.tex}

\begin{document}
%------------------------------------------------------------------

\title[Numerics for fractional elliptic SPDEs with spatial white noise]
	{Numerical solution of fractional elliptic stochastic PDEs with spatial white noise}

%------------------------------------------------------------------
%\author[D.~Bolin]{David Bolin}
%\author[K.~Kirchner]{Kristin Kirchner}
%\author[M.~Kov\'acs]{Mih\'{a}ly Kov\'{a}cs}

\author[D.~Bolin, K.~Kirchner, and M.~Kov\'{a}cs]{David Bolin, Kristin Kirchner, and Mih\'{a}ly Kov\'{a}cs}

\address[David Bolin, Kristin Kirchner, Mih\'{a}ly Kov\'{a}cs]{
Department of Mathematical Sciences\\
Chalmers University of Technology and University of Gothenburg\\
412 96 G\"oteborg, Sweden}

%\email[\emph{corresponding author}]{kristin.kirchner@chalmers.se}
\email{{\normalfont{(K.~Kirchner, corresponding author) }}kristin.kirchner@chalmers.se}
%\email{david.bolin@chalmers.se, kristin.kirchner@chalmers.se (corresponding author), mihaly@chalmers.se}

%------------------------------------------------------------------
% Acknowledgement

\thanks{
Acknowledgement.
This work was supported in part by the
Swedish Research Council (grant Nos.~2016-04187, 2017-04274),
and the Knut and Alice Wallenberg Foundation
(KAW 20012.0067).
The authors
thank Stefan A.~Funken
for his contribution
to the FEM implementations
in {\sc Matlab}.
}

%-------------------------------------------------------------------

\begin{abstract}
The numerical approximation of solutions to
stochastic partial differential equations
with additive spatial white noise
on bounded domains in $\mathbb{R}^d$ is considered.
The differential operator is given by the
fractional power~$L^\beta$, $\beta\in(0,1)$,
of an integer order elliptic differential operator $L$
and is therefore non-local.
Its inverse
$L^{-\beta}$
is represented
by a Bochner integral
from the Dunford--Taylor functional calculus.
By applying a quadrature formula
to this integral representation,
the inverse fractional power operator $L^{-\beta}$
is approximated by a weighted sum of
non-fractional
resolvents
$(I + t_j^2 L)^{-1}$
at certain quadrature nodes $t_j>0$.
The resolvents are then discretized
in space by a
standard finite element method.

This approach is
combined with an
approximation of the
white noise,
which is based only
on the mass matrix
of the finite element discretization.
In this way,
an efficient
numerical algorithm
for computing samples
of the approximate solution
is obtained.
For the resulting approximation,
the strong mean-square error is analyzed
and an explicit rate of convergence
is derived.
Numerical experiments for $L=\kappa^2-\Delta$,
$\kappa > 0$, with homogeneous Dirichlet boundary conditions
on the unit cube $(0,1)^d$
in $d=1,2,3$ spatial dimensions
for
varying $\beta\in(0,1)$
attest the theoretical results.
\end{abstract}


 %secondary: 35R60, 60G15, 60G60, 60H15

\date{\today}

\maketitle

%=======================================================================================

%==============================================================================
\section{Introduction}\label{section:intro}
%==============================================================================

A real-valued Gaussian random field $u$ defined on a spatial domain
$\cD \subset \bbR^d$ is called a Gaussian Mat\'{e}rn field
if its covariance function $C\from \cD\times \cD \to \bbR$ is given by
\begin{align}\label{e:materncov}
C(\xvec_1, \xvec_2) = \frac{2^{1-\nu}\sigma^2}{\Gamma(\nu)}
                      (\kappa \norm{\xvec_1 - \xvec_2}{})^{\nu}
                      K_{\nu}(\kappa\|\xvec_1 - \xvec_2\|),
                      \quad
                      \xvec_1, \xvec_2 \in \cD,
\end{align}
where $\norm{\cdot}{}$ is the Euclidean norm on $\bbR^d$ and
$\Gamma$, $K_{\nu}$ denote the gamma function
and the modified Bessel function of the second kind, respectively.
Via the positive
parameters $\sigma$, $\nu$, and $\kappa$
the most important characteristics of
the random field $u$ can be controlled:
its variance, smoothness, and correlation range.
Due to this flexibility,
Gaussian Mat\'ern fields
are often used for modeling
in spatial statistics, see, e.g., \cite{stein1999}.
However, a major drawback of
this traditional covariance-based representation of Mat\'ern fields
is its high computational effort.
For instance, sampling from $u$ at $n$ locations
$\xvec_1,\ldots,\xvec_n \in \cD$
requires a matrix factorization of an $n \times n$ covariance matrix
and, thus, in general, $\mathcal{O}(n^3)$ arithmetic operations.

There are several approaches attempting to
cope with this computational problem
\citep[see, e.g.,][]{banerjee2008gaussian, furrer2006covariance, nychka2015multiresolution, sun2012geostatistics}.
One of them is based
on the insight that
a Gaussian Mat\'ern field on $\cD := \bbR^d$
with parameters $\sigma,\nu,\kappa>0$
can be seen as the statistically stationary
solution $u$ to the
stochastic partial differential equation (SPDE)
\begin{equation}\label{e:statmodel}
        (\kappa^2 - \Laplace)^\beta \, u(\xvec) = \white(\xvec),
        \qquad
        \xvec\in\cD.
\end{equation}
Here, $\Laplace$ denotes the Laplacian,
$4\beta = 2\nu + d$,
and
$\white$ is Gaussian white noise on $\bbR^d$.
The marginal variance of $u$ is then given by
$\sigma^2 = \Gamma(2\beta - d/2)\Gamma(2\beta)^{-1}(4\pi)^{-d/2}\kappa^{d - 4\beta}$.

This relation between
Mat\'ern fields and SPDEs
had already been noticed
by \cite{whittle54,whittle63}.
Later on,
based on a finite element discretization of \eqref{e:statmodel},
where the differential operator $\kappa^2-\Delta$ is
augmented with Neumann boundary conditions,
Markov random field \emph{approximations}
of Mat\'ern fields
on \emph{bounded} domains $\cD\subsetneq \bbR^d$
were introduced by~\cite{lindgren11}.
Owing to the computational savings of this approach
compared to the covariance-based approximations,
these SPDE-based models have become very popular
in spatial statistics \citep[see, e.g.,][]{bhatt2015effect,lai2015spatial},
and they are still
subject of current research,
mainly because of
the following reason:
the SPDE formulation \eqref{e:statmodel} facilitates
various generalizations of approximations of Gaussian Mat\'ern fields
involving
\begin{enumerate*}[label=(\alph*)]
\item other differential operators \citep{bolin11, fuglstad2015, lindgren11},
\item more general domains, such as the sphere \citep{bolin11, lindgren11},
and
\item non-Gaussian driving noise \citep{bolin14, wallin15}.
\end{enumerate*}
However, a considerable drawback of the
finite element approximation proposed by~\cite{lindgren11}
is that it is
only computable if
$2\beta \in \bbN$.
This limits flexibility, and, in particular,
it implies that the method cannot be applied
to the important special case of
exponential covariance ($\nu = 1/2$) in $\mathbb{R}^2$,
which corresponds to $\beta = 3/4$.

With the objective of extending the approach of~\cite{lindgren11} to
all admissible values $\nu > 0$ (i.e., $\beta > d/4)$
and the generalizations mentioned above,
we consider~\eqref{e:statmodel}
in the more general framework
of fractional order elliptic equations
driven by white noise.
We propose an explicit numerical scheme
for generating samples of an approximation
to the Gaussian solution process,
which allows for any fractional power $\beta>d/4$.
This method is based on:
\begin{enumerate*}
\item\label{enum:method1} a standard finite element discretization in space,
\item\label{enum:method2} a quadrature approximation of the inverse fractional elliptic differential operator
    proposed by~\cite{bonito2015} for deterministic equations, and
\item\label{enum:method3} an approximation of the noise term on the right-hand side, whose
    covariance matrix after discretization is equal to the finite element mass matrix.
\end{enumerate*}
Due to~\ref{enum:method3}
no explicit knowledge of the eigenfunctions
of the differential operator is needed
in contrast to approximations
based on truncated Karhunen--Lo\`{e}ve expansions
of the noise term \cite[e.g.,][]{zhang2016}.

While \cite{zhang2016} remarked that any orthonormal basis
can be used in the Karhunen--Lo\`{e}ve expansion of the noise term
before projecting it onto the finite element space,
their error analysis strongly benefits from the fact
that the eigenfunctions of the differential operator are used.
In fact, if a general orthonormal basis was used,
%which also had to be constructed if the spatial domain has complex geometry,
it would be difficult to obtain explicit rates of convergence
with respect to the truncation level after truncating the expansion at a finite level.
Only if the constructed basis has certain smoothness
with respect to the differential operator,
it is possible to obtain explicit rates of convergence
\citep[see, e.g.,][addressing this kind of problem]{kovacs2011}.
Constructing such an orthonormal basis requires
a lot of computational effort, in particular, for complicated domains and $d=2,3$.
In contrast, the present approach is based on an expansion
with respect to the standard (non-orthonormal)
finite element basis for the discretized problem,
which is readily available even for complex geometries.

This work follows a series of investigations of
numerical methods for deterministic fractional order
equations \citep{baeumer2015, bonito2017, bonito2015, caffarelli2007, gavrilyuk1996, gavrilyuk2004, gavrilyuk2005, jin2015, nochetto2015, roop2006},
and for non-fractional elliptic SPDEs with random forcing
\citep[e.g.,][]{babuska2004, cao2007, du2003, gyongy2006, zhang2016}.
An essentially similar idea to our approach, which combines~\ref{enum:method3}
with Lanczos and Krylov subspace methods,
was persued by~\cite{simpson08phd}
for the differential operator $\kappa^2-\Delta$ in~\eqref{e:statmodel}.
However, neither weak nor strong convergence have been proven
and the empirical results show a generally poor performance of the proposed scheme.
A weak error estimate
for $\kappa^2 - \Delta$ and the non-fractional case $\beta=1$
in $d=2$ spatial dimensions has been derived
by~\cite{simpson2012think}, showing
quadratic weak convergence for that particular case.
Since the first two moments uniquely determine
the distribution of the Gaussian solution process,
an alternative approach is to
consider the problem of approximating the
covariance function of the solution
instead of solving for the solution process itself \citep{Dolz2017}.
Note that this method cannot be generalized to non-Gaussian models.

The structure of this article is as follows:
In \S\ref{section:main} we present
the fractional order equation
with Gaussian white noise
in a Hilbert space setting,
the necessary assumptions,
and comment on existence and uniqueness
of a solution to this problem.
Furthermore, we introduce
the numerical approximation
of the solution process
and state our main result in Theorem~\ref{thm:strong-conv}:
strong mean-square convergence
of this approximation
at an explicit rate.
This theorem is proven in \S\ref{section:error}
by partitioning the strong mean-square error
in several terms and
estimating these terms one by one.
In addition, a weak type convergence result is obtained in Corollary~\ref{cor:err-weak}.
In \S\ref{section:numexp} the
SPDE~\eqref{e:statmodel} is considered for
numerical experiments on
the unit cube $\cD = (0,1)^d$
with continuous, piecewise linear
finite element basis functions
in $d=1,2,3$ spatial dimensions.
The outcomes of the paper are summarized
in \S\ref{section:conclusion}.

%=======================================================================================
\section{Model problem and main result}\label{section:main}
%=======================================================================================

In the following, let $H$ denote a separable Hilbert space and
$L\from\scrD(L) \subset H \to H$ be a densely defined, self-adjoint, positive definite
linear operator with a compact inverse.
In this case, there exists an orthonormal basis $\{e_j\}_{j\in\bbN}$
of $H$ consisting of eigenvectors of $L$.
The eigenvalue-eigenvector pairs $\{(\lambda_j, e_j)\}_{j\in\bbN}$
can be arranged such that the
positive eigenvalues $\{\lambda_j\}_{j\in\bbN}$
are in nondecreasing order:
\begin{align*}
    0 < \lambda_1 \leq \lambda_2 \leq \ldots \leq \lambda_j \leq \lambda_{j+1}\leq \ldots,
    \qquad
    \lim_{j\to\infty} \lambda_j = \infty.
\end{align*}

We assume that the growth of the eigenvalues
is given by $\lambda_j \propto j^\alpha$
for an exponent $\alpha > 0$, i.e.,
there exist constants $c_\lambda, C_\lambda > 0$
such that
\begin{align}\label{e:lambdaj}
    c_\lambda \, j^\alpha \leq \lambda_j \leq C_\lambda \, j^\alpha
    \qquad \forall j\in\bbN.
\end{align}

For $\beta > 0$ and
$\phi \in \scrD(L^{\beta}) := \{ \psi\in H : \sum_{j\in\bbN} \lambda_j^{2\beta} \scalar{\psi, e_j}{H}^2 < \infty\}$
the action of the fractional power operator $L^{\beta} \from \scrD(L^{\beta}) \to H$
is defined by
\begin{align}\label{e:def:Lbeta}
    L^{\beta} \phi := \sum_{j\in\bbN} \lambda_j^{\beta} \scalar{\phi, e_j}{H} \, e_j.
\end{align}
The subspace $\Hdot{2\beta} := \scrD(L^{\beta}) \subset H$
is itself a Hilbert space
with respect to the inner product
and corresponding norm given by
\begin{align*}
    \scalar{\phi, \psi}{2\beta} := \scalar{L^{\beta} \phi, L^{\beta} \psi}{H},
    \qquad
    \norm{\phi}{2\beta}^2 := \norm{L^{\beta} \phi}{H}^2 = \sum_{j\in\bbN} \lambda_j^{2\beta} \scalar{\phi, e_j}{H}^2.
\end{align*}
Its dual space after identification via the inner product
on $H$ is denoted by $\Hdot{-2\beta}$.
For $s \geq 0$, the norm on the dual space $\Hdot{-s}$
enjoys the useful representation
\begin{align}\label{e:norm:H-s}
    \norm{g}{-s}
    =
    \sup\limits_{\phi\in\Hdot{s}\setminus\{0\}}
    \frac{\duality{g, \phi}{}}{\norm{\phi}{s}}
    =
    \biggl( \sum_{j\in\bbN} \lambda_j^{-s} \duality{g, e_j}{}^2 \biggr)^{1/2} ,
\end{align}
where $\duality{\cdot,\cdot}{}$ denotes
the duality pairing between $\Hdot{-s}$ and
$\Hdot{s}$ \citep[][Proof~of Lem.~5.1]{thomee2007}.
We obtain the following scale of
densely and continuously embedded Hilbert spaces:
\begin{align}\label{e:Hs-embed}
    \Hdot{s}
    \hookrightarrow
    \Hdot{r}
    \hookrightarrow
    \Hdot{0}
    :=
    H
    \cong
    \Hdot{-0}
    \hookrightarrow
    \Hdot{-r}
    \hookrightarrow
    \Hdot{-s},
    \qquad
    0\leq r \leq s.
\end{align}

It is an immediate consequence of these definitions that
$L^\beta$ is an isometric isomorphism
from $\Hdot{s}$ to $\Hdot{s-2\beta}$ for $s\geq 2\beta$,
since for $\phi \in \Hdot{s}$ it holds
\begin{align}\label{e:Lbeta:iso}
    \norm{L^\beta \phi}{s-2\beta}^2
        &= \Bigl\| \sum_{j\in\bbN} \lambda_j^\beta \scalar{\phi, e_j}{H} \, e_j \Bigr\|_{s-2\beta}^2
         = \sum_{j\in\bbN} \lambda_j^{s-2\beta} \lambda_j^{2\beta}  \scalar{\phi, e_j}{H}^2 = \norm{\phi}{s}^2.
\end{align}

The following lemma states that
$L^\beta$ can be extended to
a bounded linear operator between $\Hdot{s}$
and $\Hdot{s-2\beta}$ for all $s\in\bbR$.

\begin{lemma}\label{lem:isometry}
For $s\in\bbR$, there exists a unique continuous extension of
$L^\beta$ defined in~\eqref{e:def:Lbeta}
to an isometric isomorphism
$L^\beta \from \Hdot{s} \to \Hdot{s-2\beta}$.
\end{lemma}

\begin{proof}
For $s \geq 2\beta$ the isometry property
and, thus, injectivity has already been
observed in~\eqref{e:Lbeta:iso}.
Surjectivity readily follows,
since for any $g \in\Hdot{s-2\beta}$
the vector
$\phi:=\sum_{j\in\bbN} \lambda_j^{-\beta} \scalar{g, e_j}{H} \, e_j$
is an element of $\Hdot{s}$ with
$\norm{\phi}{s}=\norm{g}{s-2\beta}$
and $L^\beta \phi = g$.

Assume now that $s < 2\beta$. We obtain for
$\phi \in \Hdot{2\beta}$ and $\psi \in \Hdot{2\beta-s}$
\begin{align*}
    \duality{L^\beta \phi, \psi}{}
    =
    \scalar{L^\beta \phi, \psi}{H}
        = \sum_{j\in\bbN} \lambda_j^{\beta} \scalar{\phi, e_j}{H} \scalar{\psi, e_j}{H}
        \leq \norm{\phi}{s} \norm{\psi}{2\beta - s}.
\end{align*}
By density of $\Hdot{2\beta}\hookrightarrow\Hdot{s}$,
it follows that there exists a unique linear continuous extension
$L^\beta \from \Hdot{s} \to \Hdot{s-2\beta}$.
Moreover, it is an isometry, since
the above estimate attains equality
for $\phi\in\Hdot{s}$ and
$\psi = L^{s-\beta} \phi \in \Hdot{2\beta-s}$.
Surjectivity follows
similarly
to $s\geq 2\beta$
from the representation
of the dual norm
in~\eqref{e:norm:H-s}
applied to~$\Hdot{s-2\beta}$.
\end{proof}

\begin{example}\label{ex:eig:laplace}
For $\kappa \geq 0$ and a bounded, convex, polygonal domain $\cD\subset\bbR^d$,
consider the eigenvalue value problem
\begin{align*}
    \hspace*{4cm}
        (\kappa^2 - \Delta) e &= \lambda e && \text{in } \cD,  \hspace*{4cm} \\
        e &= 0 && \text{on }\partial \cD,
\end{align*}
i.e., the
operator $L = \kappa^2 - \Delta$
with homogeneous Dirichlet boundary conditions
on $H=L_2(\cD)$.
For this case, we have
$\Hdot{2}=\scrD(L) = H^2(\cD) \cap H^1_0(\cD)$,
$\Hdot{1} = H^1_0(\cD)$,
as well as
$\Hdot{-1} = H^1_0(\cD)^* = H^{-1}(\cD)$,
where$\,\,^*$ denotes the dual
after identification
via the inner product on $L_2(\cD)$.
For $s\in(0,1)$ one obtains the intermediate spaces
\begin{align*}
    \Hdot{s}  &= H^s_0(\cD), &  H^s_0(\cD) &:= [ L_2(\cD), H^1_0(\cD)]_{s,2}, \\
    \Hdot{-s} &= H^{-s}(\cD), &  H^{-s}(\cD) &:= [H^{-1}(\cD), L_2(\cD)]_{1-s,2} = H^s_0(\cD)^*,
\end{align*}
where
$[\cdot,\cdot]_{s,q}$ denotes the real $K$-interpolation method
\citep[see, e.g.,][Ch.~19]{thomee2007}.
Furthermore, one can show \citep[][Prop.~4.1]{bonito2015}
that for $1\leq s \leq 2$ it holds that
\begin{align*}
    \Hdot{s} = [\Hdot{1}, \Hdot{2}]_{s-1,2} = [H^1_0(\cD), H^2(\cD) \cap H^1_0(\cD)]_{s-1,2} = H^s(\cD) \cap H^1_0(\cD).
\end{align*}

If $\cD=(0,1)^d$ is the $d$-dimensional unit cube,
then it is well-known
\citep[e.g.,][Ch.~VI.4]{courant1962}
that the above eigenvalue problem
has the following eigenvectors
\begin{align}\label{e:eigfct:laplace}
    e_{\thetavec}(\xvec)
        = e_{(\theta_1,\ldots,\theta_d)}(x_1,\ldots,x_d)
        = \prod_{i=1}^{d} \left( \sqrt{2} \sin(\pi\theta_i x_i) \right),
\end{align}
where $\thetavec=(\theta_1,\ldots,\theta_d)\in\bbN^d$
is a $d$-dimensional multi-index.
The corresponding eigenvalues
are given by
\begin{align}\label{e:eigval:laplace}
    \lambda_{\thetavec} = \kappa^2 + \pi^2 |\thetavec|^2 = \kappa^2 + \pi^2 \sum_{i=1}^{d} \theta_i^2.
\end{align}
These eigenvalues satisfy~\eqref{e:lambdaj} for $\alpha = 2/d$.
Note that only the values of $c_\lambda$ and $C_\lambda$ in~\eqref{e:lambdaj} change
when considering any other bounded domain $\cD\subset\bbR^d$
with smooth or polygonal boundary. The value $\alpha=2/d$ is the same
by the min-max principle.
\end{example}

\subsection{Fractional order equation}\label{subsec:fraceq}

Motivated by~\eqref{e:statmodel}
we consider for
$g\in H$ the
fractional order equation
\begin{align}\label{e:Lbeta}
    L^\beta u = g + \white,
\end{align}
where $\white$ denotes Gaussian
white noise defined on a complete probability space
$(\Omega, \cA, \bbP)$ with values in the separable Hilbert space $H$.
We assume that $\beta\in(0,1)$
and refer to Remark~\ref{rem:beta>1}
for a discussion of the generalization to $\beta > 0$.
Equation~\eqref{e:Lbeta} as well as all following
equalities involving noise terms are understood to hold
$\bbP$-almost surely ($\bbP$-a.s.).

Note that the white noise $\white$ can formally be represented
by the Karhunen--Lo\`{e}ve expansion
with respect to the orthonormal eigenbasis $\{e_j\}_{j\in\bbN} \subset H$
of $L$:
\begin{align}\label{e:kl}
    \white = \sum_{j\in\bbN} \xi_j \, e_j,
\end{align}
where $\{\xi_j\}_{j\in\bbN}$ is a sequence of
independent real-valued
standard normally distributed random variables.
As we will show in Proposition~\ref{prop:whitereg},
this formally defined series
converges in $L_2(\Omega; \Hdot{-s})$
for any $s>1/\alpha$, which implies that
realizations of $\white$ are
elements of $\Hdot{-\frac{1}{\alpha}-\epsilon}$ for any $\epsilon>0$ $\bbP$-a.s.
Related to this representation, we introduce for $N\in\bbN$
the truncated white noise
\begin{align}\label{e:klN}
    \white_N := \sum_{j=1}^N \xi_j \, e_j.
\end{align}

The following proposition specifies mean-square regularity
of the white noise
with respect to the $\dot{H}$-spaces in~\eqref{e:Hs-embed}.

\begin{proposition}\label{prop:whitereg}
For all $s\geq 0$, there exists a constant
$C > 0$ depending only
on $\alpha$, $c_\lambda$
in~\eqref{e:lambdaj}, and $s$, such that
the truncated white noise $\white_N$ in~\eqref{e:klN}
satisfies
\begin{align*}
    \bbE\bigl[ \norm{\white_N}{-s}^2 \bigr]
    \leq
    C %(c_\lambda,\alpha,s)
        \begin{cases}
            1 + N^{1-\alpha s}  & s \neq \alpha^{-1}, \\
            1 + \ln(N)          & s =    \alpha^{-1}.
        \end{cases}
\end{align*}
Furthermore, it holds
$\white \in L_2\bigl(\Omega; \Hdot{-\frac{1}{\alpha} - \epsilon}\bigr)$
for all $\epsilon>0$ with
\begin{align*}
    \bbE\bigl[ \norm{\white}{-\frac{1}{\alpha} - \epsilon}^2 \bigr]
    \leq
    c_\lambda^{-\frac{1}{\alpha} - \epsilon} \left( 1 + \tfrac{1}{\epsilon\alpha} \right).
\end{align*}
\end{proposition}

\begin{proof}
The orthonormality of the vectors $\{e_j\}$
along with the distribution of the random variables $\{\xi_j\}$
in the expansion of $\white_N$
and the representation~\eqref{e:norm:H-s} of the dual norm
yield
$\bbE\bigl[ \norm{\white_N}{-s}^2 \bigr] = \sum_{j=1}^N \lambda_j^{-s}$.
Thus, we conclude with~\eqref{e:lambdaj}
\begin{align*}
    \bbE[ \norm{\white_N}{-s}^2 ] %&= \sum_{j=1}^N \lambda_j^{-s}
        \leq c_\lambda^{-s} \sum_{j=1}^N j^{-\alpha s}
        \leq c_\lambda^{-s} \Bigl( 1 + \int_1^N x^{-\alpha s} \, \rd x \Bigr)
        = c_\lambda^{-s}
        \begin{cases}
            \frac{N^{1-\alpha s} - \alpha s}{1-\alpha s} & s\neq \alpha^{-1}, \\
            1 + \ln(N) & s = \alpha^{-1}.
        \end{cases}
\end{align*}
Choosing $s=\alpha^{-1} + \epsilon$ and taking the limit $N\to\infty$ shows
the $L_2\bigl(\Omega; \Hdot{-\frac{1}{\alpha}-\epsilon}\bigr)$ regularity of $\white$.
\end{proof}

\begin{remark}\label{rem:solreg}
Proposition~\ref{prop:whitereg} implies that
$\white \in \Hdot{-\frac{1}{\alpha}-\epsilon}$
holds also $\bbP$-a.s.~for any $\epsilon>0$.
Therefore, existence and uniqueness of a solution $u$ to~\eqref{e:Lbeta}
with regularity $u\in \Hdot{2\beta - \frac{1}{\alpha} - \epsilon}$ ($\bbP$-a.s.)
follow from Lemma~\ref{lem:isometry}.
In addition, the above results show
that $u\in L_2\bigl(\Omega; \Hdot{2\beta - \frac{1}{\alpha} - \epsilon}\bigr)$.
In particular, $u\in L_2(\Omega; H)$ holds
if $2\alpha\beta > 1$.
\end{remark}

\begin{remark}\label{rem:refreg}
The above results are in accordance with Lemma~4.2 and Theorem~4.3 by~\cite{zhang2016},
where the following semilinear elliptic stochastic boundary value problem is considered
on $\cD:=(0,1)^d$ for $d\in\{1,2,3\}$, $g\in L_2(\cD)$, and $f\from\bbR\to\bbR$:
\begin{align*}
    \hspace*{2cm} -\Delta u(\xvec) + f(u(\xvec)) &= g(\xvec) + \white(\xvec), && \xvec\in \cD, \hspace*{2cm} \\
    \hspace*{2cm} u(\xvec) &= 0, && \xvec\in\partial \cD, \hspace*{2cm}
\end{align*}
and mean-square regularity in the Sobolev space $H^s(\cD)$
is proven under appropriate assumptions on the nonlinearity $f$ for
\begin{enumerate*}
    \item the spatial white noise $\white$ and $s < -d/2$ and
    \item the solution $u$ and $s < 2 - d/2$.
\end{enumerate*}
Note that in the linear case when $f\equiv \kappa^2 \geq 0$, this corresponds to
Problem \eqref{e:Lbeta} with
$\beta=1$, $L=\kappa^2-\Delta$, and the exponent $\alpha$
of the eigenvalue growth in~\eqref{e:lambdaj}
is given by $\alpha = 2/d$, see Example~\ref{ex:eig:laplace}.
\end{remark}

\subsection{Finite element approximation}\label{subsec:FEM}
In the following, we introduce a numerical method based on a finite element discretization
for solving the fractional order equation~\eqref{e:Lbeta}
approximately. For this purpose, we consider the family
$(V_h)_{h\in(0,1)}$ of
subspaces of $\Hdot{1}$
with finite dimensions
$N_h := \dim(V_h) < \infty$.
The %discrete operator of
Galerkin discretization of the operator
$L\from \Hdot{1} \to \Hdot{-1}$
is denoted
by $L_h \from V_h \to V_h$, i.e.,
$\scalar{L_h\psi_h,\phi_h}{H} = \duality{L\psi_h,\phi_h}{}$
for all $\psi_h,\phi_h\in V_h$.
For $g\in H$, the finite element approximation
of $v = L^{-1}g$ is then given by
$v_h = L_h^{-1} (\Pi_h g)$, where
$\Pi_h\from H \to V_h$ denotes the $H$-orthogonal projection
onto the finite element space $V_h$, i.e.,
\begin{align*}
    \duality{L v_h, \phi_h}{}
    =
    \scalar{L_h v_h, \phi_h}{H}
    =
    \scalar{\Pi_h g, \phi_h}{H}
    =
    \scalar{g, \phi_h}{H}
    \quad
    \forall \phi_h \in V_h.
\end{align*}
The eigenvalues $\{\lambda_{j,h}\}_{j=1}^{N_h}$
as well as the corresponding $H$-orthonormal eigenvectors
$\cE := \{e_{j,h}\}_{j=1}^{N_h}$
of $L_h$ % \from V_h \to V_h$
satisfy
the variational equalities
\begin{align}\label{e:eig-h}
    \scalar{ L_h e_{j,h}, \phi_h }{H} = \lambda_{j,h} \scalar{e_{j,h}, \phi_h}{H}
    \quad
    \forall \phi_h \in V_h,
    \
    1\leq j \leq N_h.
\end{align}
These eigenvalues are again arranged in nondecreasing order:
\begin{align*}
    0 < \lambda_{1,h} \leq \lambda_{2,h} \leq \ldots \leq \lambda_{N_h ,h}.
\end{align*}

Further assumptions
on the approximation properties
of finite element spaces
are specified below.
\begin{assumption}\label{ass:Vh}
The family $(V_h)_{h\in(0,1)}$ % \subset \Hdot{1}$
of finite-dimensional subspaces of $\Hdot{1}$
satisfies the following:
\begin{enumerate}
\item\label{ass:Vh-1} there exists $d\in\bbN$ such that $N_h = \dim(V_h) \propto h^{-d}$ for all $h > 0$;
\item\label{ass:Vh-2} there exist constants $C_1, C_2 > 0$, $h_0\in(0,1)$,
    as well as exponents $r,s > 0$, and $q>1$
    such that $\{\lambda_{j,h}\}$ and $\{e_{j,h}\}$ in~\eqref{e:eig-h} satisfy
    \begin{align}
        \lambda_j \leq \lambda_{j,h} &\leq \lambda_j + C_1 h^r \lambda_j^q, \label{e:ass-lambdajh} \\
        \norm{e_j - e_{j,h}}{H}^2    &\leq C_2 h^{2s} \lambda_j^q  \label{e:ass-ejh}
    \end{align}
    for all $h\in(0,h_0)$ and $j\in\{1,\ldots,N_h\}$.
\end{enumerate}
\end{assumption}

The following example illustrates
that Assumption~\ref{ass:Vh} is,
in general, satisfied for $d$-dimensional
elliptic linear differential operators.

\begin{example}\label{ex:ass:Vh}
Let $\cD\subset \bbR^d$ be a bounded, convex, polygonal domain
and
assume that $L\from\scrD(L)\subset L_2(\cD) \to L_2(\cD)$
is a (strongly) elliptic linear differential operator of order $2m\in\bbN$, i.e.,
there exists a constant $\gamma > 0$ such that
\begin{align*}
    \duality{Lv, v}{} \geq \gamma\norm{v}{H^m(\cD)}^2
    \quad
    \forall v\in V,
\end{align*}
where $V = H^m(\cD)$ or $V=H_0^1(\cD) \cap H^m(\cD)$,
and $\duality{\cdot, \cdot}{}$ denotes the duality
pairing between $V$ and its dual $V^*$
after identification via the inner product on $L_2(\cD)$.
Assume that $V_h \subset V$ is
an admissible finite element space
of polynomial degree $p\in\bbN$
with respect
to a regular mesh on $\bar{\cD}$.
In this case, Assumption~\ref{ass:Vh} is satisfied
\citep[][Thms.~6.1~\&~6.2]{strang2008}
for $H=L_2(\cD)$, $\Hdot{1} = V$,
and the exponents
\begin{align*}
    r = 2(p+1-m), \quad
    s = \min\{p+1, 2(p+1-m)\},  \quad \text{and} \quad
    q = \tfrac{p+1}{m}.
\end{align*}

In particular,
%if $V_h$ is the finite element space
%of continuous, piecewise linear functions,
if the family of finite element spaces $V_h$ is quasi-uniform and has
continuous, piecewise linear basis functions,
we have for $L= \kappa^2 - \Delta$ with homogeneous Dirichlet
boundary conditions
in Example~\ref{ex:eig:laplace} that
$r=s=q=2$.
\end{example}

In order to approximate the white noise $\white$
on the finite element space $V_h$
we introduce the following $V_h$-valued
random variables:
\begin{enumerate}
\item an expansion with respect to the discrete eigenbasis $\cE$:
    \begin{align}\label{e:kl-Lh}
        \white_{h}^{\cE} := \sum_{j=1}^{N_h} \xi_j \, e_{j,h},
    \end{align}
    where $\xivec := \left( \xi_1,\ldots,\xi_{N_h} \right)^{T}$
    is the vector of the first $N_h$ independent standard normally distributed random variables
    in expansion~\eqref{e:kl} of $\white$;
\item an expansion with respect to any basis $\Phi:=\{\phi_{j,h}\}_{j=1}^{N_h}$ of $V_h$:
    \begin{align}\label{e:kl-Vh}
        \white_{h}^{\Phi} := \sum_{j=1}^{N_h} \widetilde{\xi_j} \, \phi_{j,h},
    \end{align}
    where the random vector
    $\widetilde{\xivec} := ( \widetilde{\xi_1},\ldots, \widetilde{\xi}_{N_h} )^{T}$
    is given by
    \begin{align*}
        \widetilde{\xivec} = \Rmat^{-1} \xivec,
        \qquad
        R_{ij} := \scalar{e_{i,h}, \phi_{j,h}}{H},
        \quad
        1\leq i,j \leq N_h.
    \end{align*}
    It is therefore Gaussian distributed with zero mean and covariance matrix
    $\Rmat^{-1}(\Rmat^{-1})^T = (\Rmat^T \Rmat)^{-1} = \Mmat^{-1}$, where $\Mmat = ( \scalar{\phi_{i,h}, \phi_{j,h}}{H} )_{i,j=1}^{N_h}$
    is the mass matrix with respect to the basis $\Phi$ of the finite element space $V_h$.
\end{enumerate}

The following lemma shows
that the above approximations of the white noise
are equal in mean-square sense.

\begin{lemma}\label{lem:white-equiv}
The noise approximations
$\white_h^\cE$ and $\white_h^\Phi$
in~\eqref{e:kl-Lh}--\eqref{e:kl-Vh} are equivalent
in $L_2(\Omega; H)$, i.e.,
$\norm{\white_h^\cE - \white_h^\Phi}{L_2(\Omega; H)} = 0$.
\end{lemma}

\begin{proof}
Inserting the definitions~\eqref{e:kl-Lh}--\eqref{e:kl-Vh}
of $\white_h^\cE$ and $\white_h^\Phi$ yields
\begin{align*}
    \bbE\bigl[ \norm{\white_h^\cE - \white_h^\Phi}{H}^2 \bigr]
    &=
    \sum_{i=1}^{N_h} \sum_{j=1}^{N_h}
            \Bigl( \bbE[\xi_i \xi_j] \scalar{ e_{i,h}, e_{j,h} }{H}
                    + \bbE\bigl[\widetilde{\xi_i} \widetilde{\xi_j}\bigr] \scalar{ \phi_{i,h}, \phi_{j,h} }{H}
    \Bigr) \\
    &\qquad
    - 2 \sum_{i=1}^{N_h} \sum_{j=1}^{N_h} \bbE\bigl[\widetilde{\xi_i} \xi_j\bigr] \scalar{ \phi_{i,h}, e_{j,h} }{H}.
\end{align*}
By definition of the matrices $\mathbf{R}$ and $\mathbf{M}$
and due to the
distribution of the random vectors $\xivec$ and $\widetilde{\xivec}$
we have for $i,j\in\{1,\ldots,N_h\}$:
\begin{align*}
    \bbE[\xi_i \xi_j]                                   &= \delta_{ij}, &
    \scalar{ e_{i,h}, e_{j,h} }{H}                      &= \delta_{ij}, \\
    \bbE\bigl[\widetilde{\xi_i} \widetilde{\xi_j}\bigr] &= [\Mmat^{-1}]_{ij}, &
    \scalar{ \phi_{i,h}, \phi_{j,h} }{H}                &= M_{ji}, \\
    \bbE\bigl[\widetilde{\xi_i} \xi_j\bigr]             &= [\Rmat^{-1}]_{ij}, &
    \scalar{ \phi_{i,h}, e_{j,h} }{H}                   &= R_{ji},
\end{align*}
where $\delta_{ij}$ is the Kronecker delta.
Thus, in terms of the trace $\tr$ of matrices
we have
\begin{align*}
    \bbE\bigl[ \norm{\white_h^\cE - \white_h^\Phi}{H}^2 \bigr]
    &=
    \tr(\Imat)
    + \tr( \Mmat^{-1} \Mmat )
    - 2 \tr( \Rmat^{-1} \Rmat )
    = 0,
\end{align*}
which proves the equivalence of
$\white_h^\cE$ and $\white_h^\Phi$
in $L_2(\Omega;H)$.
\end{proof}

Our numerical approach to cope with the fractional
order equation~\eqref{e:Lbeta} will be based on the
following representation of the inverse $L^{-\beta}$
from the Dunford--Taylor calculus
due to~\cite{balakrishnan1960}, see also \citep[][\S IX.11, Eq.~(4)]{yosida1995}:
\begin{align*}
    L^{-\beta}
        &=
        \frac{\sin(\pi\beta)}{\pi}
        \int_0^{\infty} \lambda^{-\beta} (\lambda I + L)^{-1} \, \rd \lambda \\
%        &=
%        \frac{2\sin(\pi\beta)}{\pi}
%        \int_0^{\infty} t^{2\beta-1} (I + t^2 L)^{-1} \, \rd t \\
        &=
        \frac{2\sin(\pi\beta)}{\pi}
        \int_{-\infty}^{\infty} e^{2\beta y} (I + e^{2y} L)^{-1} \, \rd y.
\end{align*}

We choose an equidistant grid
$\{y_\ell = \ell k : \ell\in\bbZ, -K^{-} \leq \ell \leq K^+\}$
with step size $k>0$ for $y$
and replace the differential
operator $L$ with its discrete version $L_h$
to formulate the following
quadrature method
proposed by~\cite{bonito2015}:
\begin{align}\label{e:Qhbeta}
    Q^\beta_{h,k} := \frac{2 k \sin(\pi\beta)}{\pi} \sum_{\ell=-K^{-}}^{K^{+}} e^{2\beta y_\ell} (I + e^{2 y_\ell} L_h)^{-1}.
\end{align}
Exponential convergence
of order $\cO(e^{-\pi^2/(2k)})$
to $L_h^{-\beta}$
with respect to the norm
\begin{align}\label{e:norm-LVh}
    \norm{T}{\cL(V_h)} := \sup\limits_{\phi_h \in V_h\setminus\{0\}} \frac{\norm{T\phi_h}{H}}{\norm{\phi_h}{H}}
\end{align}
on the space
$\cL(V_h) := \{ T \from V_h \to V_h \text{ linear}\}$
has been proven~\citep[][Lem.~3.4, Rem.~3.1, Thm.~3.5]{bonito2015}
for the choice
\begin{align*}
    K^- := \left\lceil \frac{\pi^2}{4\beta k^2} \right\rceil,
    \qquad
    K^+ := \left\lceil \frac{\pi^2}{4(1-\beta)k^2} \right\rceil.
\end{align*}
Note that the quadrature~\eqref{e:Qhbeta}
involves only non-fractional resolvents of the discrete operator $L_h$.
Thus, the corresponding numerical method is readily implementable.

With the notions of the
$H$-orthogonal projection $\Pi_h$
on $V_h$,
the noise approximation
$\white_h^\Phi$ in~\eqref{e:kl-Vh}
and the quadrature $Q_{h,k}^\beta$
in~\eqref{e:Qhbeta}
at hand,
we can now
introduce the numerical approximation $u_{h,k}^{Q}$
of the solution $u$ to~\eqref{e:Lbeta} as
\begin{align}\label{e:uQh}
    u_{h,k}^{Q} := Q_{h,k}^\beta(\Pi_h g + \white_h^\Phi).
\end{align}

\begin{remark}\label{rem:rhs-easysample}
We emphasize the construction
of the noise term $\white_h^\Phi$
on the right-hand side of~\eqref{e:uQh}:
For discretizing the white noise $\white$,
it is common to project the truncated
Karhunen--Lo\`{e}ve expansion $\white_N$ in~\eqref{e:klN}
onto the finite element space $V_h$ and use
the noise approximation $\Pi_h \white_N \in L_2(\Omega; V_h)$
\citep[e.g.,][]{zhang2016}.
Instead, we define $\white_h^\Phi$ in~\eqref{e:kl-Vh}
in such a way that it is mean-square equivalent to
the noise approximation $\white_h^\cE$ in~\eqref{e:kl-Lh},
which can be interpreted as
$V_h$-valued white noise expanded
with respect to the
$H$-orthonormal eigenvectors %$\{e_{j,h}\}_{j=1}^{N_h}$
of $L_h$.
Note that the noise approximation $\white_h^\cE$
is needed only for the error analysis, but not
for the actual implementation of the numerical algorithm.

This approach has the following
two major advantages in practice:
\begin{enumerate}
\item Samples of the truncated white noise $\white_N$
    are not needed and, thus, neither are
    the exact eigenvectors of the operator $L$.
\item For the computation of the approximation
    $u_{h,k}^{Q}$ in~\eqref{e:uQh} in practice,
    one has to sample from the load vector $\bvec$
    with entries
    \begin{align*}
        b_i := \scalar{\Pi_h g + \white_h^\Phi, \phi_{i,h}}{H},
        \qquad
        1\leq i \leq N_h,
    \end{align*}
    where
    $\{\phi_{j,h}\}_{j=1}^{N_h}$
    is a basis
    of the finite element space $V_h$.
    We emphasize that this is computationally feasible
    if the basis $\Phi$ in the
    noise approximation $\white_h^\Phi$ is chosen
    as the same, since then
    $\bvec \sim \cN(\gvec, \Mmat)$,
    where $g_i = \scalar{g, \phi_{i,h}}{H}$ and
    $\Mmat$ is the mass matrix with respect
    to the basis $\Phi$,
    which is usually sparse.
    Hence, samples of $\bvec$
    can be generated from
    $\bvec = \gvec +  \mathbf{G}\zvec$,
    where $\zvec \sim \cN(\mathbf{0},\Imat)$ and $\mathbf{G}$
    is a Cholesky factor of $\Mmat = \mathbf{G}\mathbf{G}^T$.
\end{enumerate}
\end{remark}

The following estimate of the strong mean-square error
between the exact solution $u$
to the fractional order equation~\eqref{e:Lbeta}
and the numerical approximation $u^Q_{h,k}$ in~\eqref{e:uQh}
is our main result.

\begin{theorem}\label{thm:strong-conv}
Let the family $(V_h)_{h\in(0,1)}$
of finite element spaces
satisfy Assumption~\ref{ass:Vh}
and assume that
the growth of the eigenvalues
of the operator $L$ is given
by~\eqref{e:lambdaj}
for an exponent $\alpha$ with
\begin{align}\label{e:ass-alpha}
    \tfrac{1}{2\beta}
    <
    \alpha
    \leq
    \min\left\{ \tfrac{r}{(q-1)d}, \tfrac{2s}{qd} \right\},
\end{align}
where the values of
$d\in\bbN$, $r,s>0$, and $q>1$
are the same as in Assumption~\ref{ass:Vh}.
Then, for sufficiently small $h\in(0,h_0)$ and $k\in(0,k_0)$,
the strong $L_2(\Omega; H)$ error
between the solution $u$ of~\eqref{e:Lbeta}
and the approximation $u_{h,k}^{Q}$ in~\eqref{e:uQh}
is bounded by
\begin{align}\label{e:err-strong}
    \norm{u - u_{h,k}^{Q}}{L_2(\Omega; H)}
        \leq C
        \left(h^{\min\{d(\alpha\beta-1/2),r,s\}} + e^{-\pi^2/(2k)} h^{-d/2} \right)
        (1 + \norm{g}{H}),
\end{align}
where the constant $C>0$ is independent of $h$ and $k$.
\end{theorem}

%=======================================================================================
\section{Partition of the error and error estimates}\label{section:error}
%=======================================================================================

In order to prove Theorem~\ref{thm:strong-conv}
we express the difference between the exact solution $u$
and the approximation $u_{h,k}^{Q}$ as follows:
\begin{align*}
    u - u_{h,k}^{Q}
        =
        (u - u_{N_h})
        +
        (u_{N_h} - u_h^\cE)
        +
        (u_h^\cE - u_h^\Phi)
        +
        (u_h^\Phi - u_{h,k}^{Q}),
\end{align*}
and partition the strong error in~\eqref{e:err-strong}
accordingly.
Here, $u_{N_h}$, $u_h^\cE$, and
$u_h^\Phi$ are defined
in terms of the truncated white noise
in~\eqref{e:klN}
and the white noise approximations
in~\eqref{e:kl-Lh}--\eqref{e:kl-Vh}
as the $\bbP$-almost sure solutions
to the following equations:
\begin{align}
        L^{\beta} u_{N}      &= g_{N} + \white_{N}, \label{e:uN}
\intertext{for $N\in\bbN$, where $g_{N}  := \sum_{j=1}^{N} \scalar{g, e_j}{H} \, e_j$,}
        L_h^{\beta} u_h^\cE  &= \Pi_h g + \white_h^\cE, \label{e:uhcE} \\
        L_h^{\beta} u_h^\Phi &= \Pi_h g + \white_h^\Phi, \label{e:uhPhi}
\end{align}
and $u_{N_h}$ refers to the truncation with $N=N_h=\dim(V_h)$ terms
in~\eqref{e:uN}.
Note that by Lemma~\ref{lem:white-equiv}
the difference between $u_h^\cE$ and $u_h^\Phi$ vanishes identically
in $L_2(\Omega;H)$:
\begin{align}\label{e:uh-equiv}
    \norm{u_h^\cE - u_h^\Phi}{L_2(\Omega;H)}
        \leq \norm{L_h^{-\beta}}{\cL(V_h)} \norm{\white_h^\cE - \white_h^\Phi}{L_2(\Omega;H)} = 0.
\end{align}

In the following, we address
the three remaining terms separately:
the truncation error,
the error of the finite element discretization,
and the error caused by the quadrature approximation $Q^\beta_{h,k}$ in~\eqref{e:Qhbeta}
of the discrete fractional inverse $L_h^{-\beta}$.

%=======================================================================================
\subsection{Truncation error}\label{subsec:err-trunc}
%=======================================================================================

In the lemma below, we bound the strong mean-square error
between the exact solution $u$ to the fractional
order equation~\eqref{e:Lbeta}
and the approximation $u_N$ in~\eqref{e:uN}
based on the truncated right-hand side $g_N + \white_N$.

\begin{lemma}\label{lem:err-trunc}
Assume that the eigenvalues of
the operator $L$ satisfy~\eqref{e:lambdaj}
and that $2\alpha\beta > 1$.
Let $u \in L_2(\Omega; H)$
be the solution to~\eqref{e:Lbeta}.
Then, for any $g\in H$, $N\in\bbN$, there exists a
unique solution $u_N\in L_2(\Omega; \Hdot{2\beta})$ to
the truncated equation~\eqref{e:uN} and it
satisfies
\begin{align*}
    \bbE\left[ \norm{u-u_N}{H}^2 \right]
        \leq C N^{-(2\alpha\beta-1)} \left(1 + \norm{g}{H}^2\right),
\end{align*}
where the constant $C>0$ depends only on
$\alpha, \beta$, and the constants in~\eqref{e:lambdaj}.
In particular,
under Assumption~\ref{ass:Vh}
it holds that %for $N=N_h$
\begin{align}\label{e:err-trunc}
    \norm{u-u_{N_h}}{L_2(\Omega;H)}
        \lesssim h^{d(\alpha\beta-1/2)} (1 + \norm{g}{H}).
\end{align}
\end{lemma}

\begin{proof}
Existence and uniqueness of a solution $u_N$
in $L_2(\Omega; \Hdot{2\beta})$ follows from
the fact that for any $g\in H$ the truncated
right-hand side $g_N + \white_N$ is an
element of $L_2(\Omega; H)$
as well as
the isomorphism property of
$L^\beta\from\Hdot{2\beta}\to H$
in Lemma~\ref{lem:isometry}.
If $2\alpha\beta > 1$, we obtain for any $N\in\bbN$:
\begin{align*}
    \bbE\left[ \norm{u-u_N}{H}^2 \right]
        &= \bbE\bigl[ \norm{L^{-\beta} ( g - g_N + \white - \white_N )}{H}^2 \bigr]
         = \sum_{j>N} \lambda_j^{-2\beta} \left( \scalar{g,e_j}{H}^2 + 1 \right) \\
        &\leq c_\lambda^{-2\beta} \left(1 + \norm{g}{H}^2 \right) \sum_{j>N} j^{-2\alpha\beta}
         \leq \frac{N^{1-2\alpha\beta}}{c_\lambda^{2\beta} (2\alpha\beta - 1)} \left(1 + \norm{g}{H}^2\right).
\end{align*}
Under Assumption~\ref{ass:Vh}\ref{ass:Vh-1} we have $N_h \propto h^{-d}$ and \eqref{e:err-trunc} readily follows.
\end{proof}

%=======================================================================================
\subsection{Finite element discretization error}\label{subsec:err-fem}
%=======================================================================================

The next ingredient in the derivation of~\eqref{e:err-strong}
in Theorem~\ref{thm:strong-conv}
is an upper bound for the error
caused by introducing the
finite element element discretization.

More precisely,
the solution $u_N$ of the truncated problem~\eqref{e:uN}
corresponds to
the best $\widetilde{V}_N$-valued approximation
of $u\in L_2(\Omega; H)$.
Here, the finite-dimensional
subspace $\widetilde{V}_N \subset \Hdot{1}$
is the linear span
of the first $N$ eigenvectors
of the operator $L$.
Subsequently,
the approximation $u_h^\cE$
has been defined
in~\eqref{e:uhcE},
which takes values in
another finite-dimensional subspace,
namely in the
finite element space $V_h \subset \Hdot{1}$.

The purpose of the following lemma
is to bound the error
between these two approximations
when $N=\dim(V_h)$.

\begin{lemma}\label{lem:err-fem}
Let Assumption~\ref{ass:Vh} be satisfied. Assume
that the eigenvalue growth of the operator $L$
is given by~\eqref{e:lambdaj}
for an exponent $\alpha$
with~\eqref{e:ass-alpha}.
Let $u_h^\cE$
be the unique element in $L_2(\Omega; V_h)$
satisfying~\eqref{e:uhcE}, i.e.,
\begin{align*}
    \scalar{L_h u_h^\cE, \phi_h}{H} = \scalar{\Pi_h g + \white_h^\cE, \phi_h}{H}
    \qquad
    \forall \phi_h \in V_h,
    \
    \bbP\text{\normalfont{-a.s.},}
\end{align*}
and let $u_{N_h} \in L_2(\Omega;H)$ denote the solution
to the truncated equation~\eqref{e:uN}
with $N=N_h=\dim(V_h)$ terms.
Then their difference
can be bounded by
\begin{align}\label{e:err-fem}
    \norm{u_{N_h} - u_h^\cE }{L_2(\Omega;H)}
        \leq C h^{\min\{ d(\alpha\beta-1/2), r, s\}} (1 + \norm{g}{H})
\end{align}
for sufficiently small $h\in(0,h_0)$,
where the constant $C>0$ depends only on
$\alpha, \beta$, and the constants
in~\eqref{e:lambdaj}, \eqref{e:ass-lambdajh}, and~\eqref{e:ass-ejh}.
\end{lemma}

\begin{proof}%[Proof of Lemma~\ref{lem:err-fem}]
In order to show the estimate in~\eqref{e:err-fem},
we first note that $\Pi_h g$ can be expanded in terms of the
orthonormal eigenvectors $\{e_{j,h}\}$ of $L_h$ by
\begin{align*}
    \Pi_h g = \sum_{j=1}^{N_h} \scalar{\Pi_h g, e_{j,h}}{H} \, e_{j,h} = \sum_{j=1}^{N_h} \scalar{g, e_{j,h}}{H} \, e_{j,h}.
\end{align*}
Thus, we can partition the difference in~\eqref{e:err-fem}
as follows: {\allowdisplaybreaks
\begin{align*}
    \norm{u_{N_h} &- u_h^\cE }{L_2(\Omega;H)}
         = \bigl\| L^{-\beta}(g_{N_h} + \white_{N_h}) - L_h^{-\beta}(\Pi_h g + \white_h^\cE) \bigr\|_{L_2(\Omega;H)} \\
        &\hspace*{-0.0133cm}= \Bigl\| \sum_{j=1}^{N_h} \lambda_j^{-\beta} (\scalar{g, e_j}{H} + \xi_j) e_j
            - \sum_{j=1}^{N_h} \lambda_{j,h}^{-\beta} (\scalar{g, e_{j,h}}{H} + \xi_j) e_{j,h} \Bigr\|_{L_2(\Omega;H)}
            \displaybreak \\
        &\hspace*{-0.0133cm}\leq \Bigl\| \sum_{j=1}^{N_h} \lambda_j^{-\beta} (\scalar{g, e_j}{H} + \xi_j) (e_j - e_{j,h}) \Bigr\|_{L_2(\Omega;H)}
            + \Bigl\| \sum_{j=1}^{N_h} \lambda_j^{-\beta} \scalar{g, e_j - e_{j,h}}{H} \, e_{j,h} \Bigr\|_{H} \\
        &\hspace*{-0.0133cm}\quad + \Bigl\| \sum_{j=1}^{N_h} (\lambda_j^{-\beta} - \lambda_{j,h}^{-\beta} ) ( \scalar{g, e_{j,h}}{H} + \xi_j) e_{j,h} \Bigr\|_{L_2(\Omega;H)}
        =: \text{\normalfont{(I)}} + \text{\normalfont{(II)}} + \text{\normalfont{(III)}}.
\end{align*}}
Since the random variables $\{\xi_j\}$ are
independent and standard normally distributed,
we obtain for the first term by the
Cauchy--Schwarz inequality for sums
\begin{align*}
    \text{\normalfont{(I)}}^2 &= \Bigl\| \sum_{j=1}^{N_h} \lambda_j^{-\beta} \scalar{g, e_j}{H} (e_j - e_{j,h}) \Bigr\|_H^2
            + \sum_{j=1}^{N_h} \lambda_j^{-2\beta} \norm{e_j-e_{j,h}}{H}^2 \\
        &\leq \Bigl( 1 + \sum_{j=1}^{N_h} \scalar{g, e_j}{H}^2 \Bigr) \sum_{j=1}^{N_h} \lambda_j^{-2\beta} \norm{e_j-e_{j,h}}{H}^2
         \leq ( 1 + \norm{g}{H}^2 ) \sum_{j=1}^{N_h} \lambda_j^{-2\beta} \norm{e_j-e_{j,h}}{H}^2 .
\end{align*}
Assumption~\ref{ass:Vh}\ref{ass:Vh-1}, \eqref{e:lambdaj}, and~\eqref{e:ass-ejh}
complete the estimation of the first term
\begin{align*}
\text{\normalfont{(I)}}^2
        &\lesssim h^{2s} (1 + \norm{g}{H}^2) \sum_{j=1}^{N_h} \lambda_j^{q-2\beta}
         \lesssim h^{2s} (1 + \norm{g}{H}^2) \sum_{j=1}^{N_h} j^{\alpha(q-2\beta)} \\
        &\lesssim h^{2s} \bigl(1 + N_h^{1 + \alpha(q-2\beta)} \bigr) (1 + \norm{g}{H}^2)
         \lesssim \bigl( h^{2s} + h^{2d(\alpha\beta-1/2)} \bigr) (1 + \norm{g}{H}^2)
\end{align*}
for $h\in(0,h_0)$, since
since $d \alpha q \leq 2s$ by~\eqref{e:ass-alpha}.
For the second term we find
\begin{align*}
    \text{\normalfont{(II)}}^2
        &= \sum_{j=1}^{N_h} \lambda_j^{-2\beta} \scalar{g, e_j - e_{j,h}}{H}^2
         \leq \norm{g}{H}^2 \sum_{j=1}^{N_h} \lambda_j^{-2\beta} \norm{e_j-e_{j,h}}{H}^2
\end{align*}
and conclude as for (I) that
$\text{\normalfont{(II)}}^2 \lesssim \bigl( h^{2s} + h^{2d(\alpha\beta-1/2)} \bigr) \norm{g}{H}^2$.
For (III) we use again the independence and distribution
of the random variables $\{\xi_j\}$ and obtain
\begin{align*}
    \text{\normalfont{(III)}}^2
        &= \sum_{j=1}^{N_h} \left( \lambda_j^{-\beta} - \lambda_{j,h}^{-\beta} \right)^2 \left( \scalar{g, e_{j,h}}{H}^2 + 1 \right)
        \leq \left( 1 + \norm{g}{H}^2 \right) \sum_{j=1}^{N_h} \left(\lambda_j^{-\beta} - \lambda_{j,h}^{-\beta} \right)^2.
\end{align*}
By the mean value theorem,
there exists $\widetilde{\lambda}_j \in (\lambda_j, \lambda_{j,h})$ such that
\begin{align*}
    \lambda_j^{-\beta} - \lambda_{j,h}^{-\beta}
        = \beta \widetilde{\lambda}_j^{-\beta-1} (\lambda_{j,h} - \lambda_j )
        \leq \beta \lambda_j^{-\beta-1} (\lambda_{j,h} - \lambda_j )
        \quad
        \forall j \in \{1,\ldots,N_h\}.
\end{align*}
Therefore, we can bound the third term
by~\eqref{e:lambdaj} and~\eqref{e:ass-lambdajh} as follows:
\begin{align*}
    \text{\normalfont{(III)}}^2
        &\lesssim ( 1 + \norm{g}{H}^2 ) \sum_{j=1}^{N_h} \lambda_j^{-2\beta-2} (\lambda_{j,h} - \lambda_j)^2
         \lesssim h^{2r} ( 1 + \norm{g}{H}^2 ) \sum_{j=1}^{N_h} \lambda_j^{2(q-\beta-1)} \\
        &\lesssim h^{2r} ( 1 + \norm{g}{H}^2 ) \sum_{j=1}^{N_h} j^{2\alpha(q-\beta-1)}
        \lesssim \bigl(h^{2r}  + h^{2d(\alpha\beta - 1/2)} \bigr) ( 1 + \norm{g}{H}^2 ),
\end{align*}
where we used the relations $N_h \propto h^{-d}$
from Assumption~\ref{ass:Vh}\ref{ass:Vh-1}
and $\alpha d (q-1) \leq r$
from~\eqref{e:ass-alpha} in the last estimate.
\end{proof}

%=======================================================================================
\subsection{Quadrature approximation error}\label{subsec:err-quad}
%=======================================================================================

As the final step for proving
the strong convergence result
of Theorem~\ref{thm:strong-conv},
we investigate the error
caused by applying the
quadrature $Q^\beta_{h,k}$
in~\eqref{e:Qhbeta}
instead of the discrete
fractional inverse $L_h^{-\beta}$
to the right-hand side
$\Pi_h g + \white_h^\Phi$.
The difference between
these two operators in
the norm~\eqref{e:norm-LVh}
on the space $\cL(V_h)$
has been bounded
by~\cite{bonito2015}.
The following lemma is a consequence
of that result and the distribution
of the noise approximation $\white_h^\Phi$.

\begin{lemma}\label{lem:err-quad}
Let $Q^\beta_{h,k} \from V_h \to V_h$ be the operator
in~\eqref{e:Qhbeta} approximating $L_h^{-\beta}$.
Then it holds for sufficiently small $k\in(0,k_0)$ and any
$\phi_h \in V_h$ that
\begin{align}\label{e:err-quad-1}
    \norm{(Q^\beta_{h,k} - L_h^{-\beta})\phi_h}{H}
        \leq C
                e^{-\pi^2/(2k)} \norm{\phi_h}{H},
\end{align}
where the constant $C>0$ depends on $\beta$ and grows linearly in
the largest eigenvalue $\lambda_{1,h}^{-1}$ of $L_h^{-1}$.
%
%In particular, the difference between
%$u_h^\Phi$ in~\eqref{e:uhPhi} and
%$u_{h,k}^Q$ in~\eqref{e:uQh}
%can be bounded as follows:
In particular, for
$u_h^\Phi$ in~\eqref{e:uhPhi} and
$u_{h,k}^Q$ in~\eqref{e:uQh},
it holds
\begin{align}\label{e:err-quad}
    \norm{u_h^\Phi - u_{h,k}^Q}{L_2(\Omega; H)}
        \lesssim
            e^{-\pi^2/(2k)} ( h^{-d/2} + \norm{g}{H} ).
\end{align}
\end{lemma}

\begin{proof}
The first assertion is proven in~\citep[][Lem.~3.4, Thm.~3.5]{bonito2015}.
This bound,
$\norm{\Pi_h g}{H} \leq \norm{g}{H}$, and
$\bbE[\norm{\white_h^\Phi}{H}^2] = N_h \lesssim h^{-d}$
imply~\eqref{e:err-quad}.
\end{proof}

%=======================================================================================
\subsection{Proof of Theorem~\ref{thm:strong-conv} and some remarks}\label{subsec:proof-remarks}
%=======================================================================================

After having bounded the truncation, the discretization, and the quadrature errors
in \S\S\ref{subsec:err-trunc}--\ref{subsec:err-quad},
the strong convergence result of Theorem~\ref{thm:strong-conv}
%for the numerical approximation $u_{h,k}^Q$ in~\eqref{e:uQh}
%of the solution $u$ to~\eqref{e:Lbeta}
is an immediate consequence.

\begin{proof}[Proof of Theorem~\ref{thm:strong-conv}]
As outlined at the beginning of the section,
we partition the mean-square error as follows:
\begin{align*}
    \norm{u - u_{h,k}^Q}{L_2(\Omega; H)}
        &\leq \norm{u - u_{N_h}}{L_2(\Omega; H)}
            + \norm{u_{N_h} - u_h^\cE}{L_2(\Omega; H)} \\
        &\qquad + \norm{u_h^\cE - u_h^\Phi}{L_2(\Omega; H)}
            + \norm{u_h^\Phi - u_{h,k}^Q}{L_2(\Omega; H)} \\
        &=: \text{(I)} + \text{(II)} + \text{(III)} + \text{(IV)}.
\end{align*}
By~\eqref{e:err-trunc}, \eqref{e:err-fem}, and \eqref{e:err-quad}
of the Lemmata~\ref{lem:err-trunc}--\ref{lem:err-quad} we have
\begin{align*}
    \text{(I)}  &\lesssim h^{d(\alpha\beta-1/2)} (1 + \norm{g}{H}), \\
    \text{(II)} &\lesssim h^{\min\{d(\alpha\beta-1/2),r,s\}} (1 + \norm{g}{H}), \\
    \text{(IV)} &\lesssim e^{-\pi^2/(2k)} (h^{-d/2} + \norm{g}{H}) \lesssim e^{-\pi^2/(2k)} h^{-d/2} (1 + \norm{g}{H}),
\end{align*}
and by~\eqref{e:uh-equiv} the third term vanishes, $\text{(III)} = 0$.
Thus, the assertion is proven.
\end{proof}

In \S\ref{section:numexp} we will
not only verify the above derived rate of strong convergence
by means of numerical experiments, % in $\bbR^d$ for $d = 1,2,3$,
but also investigate weak type errors.
The following result is proven similarly to
Theorem~\ref{thm:strong-conv}.

\begin{corollary}\label{cor:err-weak}
Let the assumptions of Theorem~\ref{thm:strong-conv} be satisfied.
Then there exists a constant $C>0$ independent of $h\in(0,h_0)$ and $k\in(0,k_0)$
such that the following weak type error estimate
between the approximation $u_{h,k}^Q$ in~\eqref{e:uQh}
and the solution $u$ to~\eqref{e:Lbeta} holds:
\begin{align}\label{e:err-weak}
    \bigl| \norm{u}{L_2(\Omega; H)}^2 &- \norm{u_{h,k}^Q}{L_2(\Omega;H)}^2 \bigr|
         \leq C \bigl (h^{\min\{ d(2\alpha\beta-1), r, s\}} + e^{- \pi^2/(2k) } \bigr) \norm{g}{H}^2  \\
        &\quad + C \bigl( h^{\min\{ d(2\alpha\beta-1), r \}}
            + e^{- \pi^2 / k } h^{-d} + e^{- \pi^2 / (2k)} +  e^{- \pi^2 / (2k)} f_{\alpha,\beta}(h) \bigr), \notag
\end{align}
where $f_{\alpha,\beta}(h) := h^{d(\alpha\beta-1)}$, if $\alpha\beta \neq 1$,
and $f_{\alpha,\beta}(h) := |\ln(h)|$, if $\alpha\beta = 1$.
\end{corollary}

\begin{proof}
First we note that the norm on $L_2(\Omega;H)$
attains the following values for
$u$ in~\eqref{e:Lbeta},
$u_{N_h}$ in~\eqref{e:uN}, and $u_h^\Phi$ in~\eqref{e:uhPhi}:
\begin{align*}
    \norm{u}{L_2(\Omega;H)}^2
        &= \sum_{j\in\bbN} \lambda_j^{-2\beta} ( 1 + \scalar{g,e_j}{H}^2), \\
    \norm{u_{N_h}}{L_2(\Omega;H)}^2
        &= \sum_{j=1}^{N_h} \lambda_j^{-2\beta} ( 1 + \scalar{g,e_j}{H}^2), \\
    \norm{u_h^\Phi}{L_2(\Omega;H)}^2
        &= \sum_{j=1}^{N_h} \lambda_{j,h}^{-2\beta} ( 1 + \scalar{g,e_{j,h}}{H}^2)
         = \norm{L_h^{-\beta} \Pi_h g}{H}^2 + \norm{L_h^{-\beta} \white_h^\cE}{L_2(\Omega;H)}^2.
%         = \norm{u_h^\cE}{L_2(\Omega;H)}^2.
%    \norm{u_{h,k}^Q}{L_2(\Omega;H)}^2
%        &= \norm{Q_{h,k}^{\beta} \Pi_h g}{H}^2 + \norm{Q_{h,k}^{\beta} \white_h^\Phi}{L_2(\Omega;H)}^2
%         = \norm{Q_{h,k}^{\beta} \Pi_h g}{H}^2 + \norm{Q_{h,k}^{\beta} \white_h^\cE}{L_2(\Omega;H)}^2.
\end{align*}
Again, we partition the error,
\begin{align*}
    \bigl| \norm{u}{L_2(\Omega; H)}^2 &- \norm{u_{h,k}^Q}{L_2(\Omega;H)}^2 \bigr|
        \leq \bigl| \norm{u}{L_2(\Omega; H)}^2 - \norm{u_{N_h}}{L_2(\Omega;H)}^2 \bigr| \\
        &\quad + \bigl| \norm{u_{N_h}}{L_2(\Omega; H)}^2 - \norm{u_h^\Phi}{L_2(\Omega;H)}^2 \bigr|
          + \bigl| \norm{u_h^\Phi}{L_2(\Omega; H)}^2 - \norm{u_{h,k}^Q}{L_2(\Omega;H)}^2 \bigr| \\
        &=: \text{(I)} + \text{(II)} + \text{(III)},
\end{align*}
and bound every term separately.
By applying similar steps as in the proofs
of Lemmata~\ref{lem:err-trunc}--\ref{lem:err-fem},
we obtain for the first two terms:
\begin{align*}
    \text{(I)} &=
        \sum_{j>N_h} \lambda_j^{-2\beta} (1 + \scalar{g,e_j}{H}^2 )
        \lesssim h^{d(2\alpha\beta-1)} (1 + \norm{g}{H}^2), \\
    \text{(II)} &=
        \sum_{j=1}^{N_h} (\lambda_j^{-2\beta} - \lambda_{j,h}^{-2\beta}) ( 1 + \scalar{g,e_{j,h}}{H}^2 )
            + \sum_{j=1}^{N_h} \lambda_j^{-2\beta} \scalar{g,e_{j} - e_{j,h}}{H} \scalar{g,e_j + e_{j,h}}{H} \\
        &\leq 2\beta (1 + \norm{g}{H}^2) \sum_{j=1}^{N_h} \lambda_j^{-2\beta-1} (\lambda_{j,h} - \lambda_j)
            +  2 \norm{g}{H}^2 \sum_{j=1}^{N_h} \lambda_j^{-2\beta} \norm{e_{j} - e_{j,h}}{H} \\
        &\lesssim h^r (1 + \norm{g}{H}^2) \sum_{j=1}^{N_h} \lambda_j^{q-2\beta-1}
            +  h^s \norm{g}{H}^2 \sum_{j=1}^{N_h} \lambda_j^{q/2 - 2\beta} \\
        &\lesssim h^{\min\{ d(2\alpha\beta-1), r \}} (1 + \norm{g}{H}^2)
            +  h^{\min\{ d(2\alpha\beta-1), s\}} \norm{g}{H}^2,
\end{align*}
since $d \alpha (q-1) \leq r$ and $d\alpha q/2 \leq s$ as assumed in~\eqref{e:ass-alpha}.
For the third term we find
\begin{align*}
    \text{(III)} &\leq
        \bigl| \norm{Q_{h,k}^\beta \Pi_h g}{H}^2 - \norm{L_h^{-\beta} \Pi_h g}{H}^2 \bigr|
            + \bigl| \norm{Q_{h,k}^{\beta} \white_h^\cE}{L_2(\Omega;H)}^2 - \norm{L_h^{-\beta} \white_h^\cE}{L_2(\Omega;H)}^2 \bigr| \\
       &=: \text{(IIIa)} + \text{(IIIb)}.
\end{align*}
Since $\norm{L_h^{-\beta}}{\cL(V_h)} = \lambda_{1,h}^{-\beta}$ and
$\norm{Q_{h,k}^\beta}{\cL(V_h)} \lesssim \lambda_{1,h}^{-\beta}$
for sufficiently small $k\in(0,k_0)$, we conclude
with~\eqref{e:err-quad-1} of Lemma~\ref{lem:err-quad} that
\begin{align*}
    \text{(IIIa)}
        &= \bigl| \scalar{(Q_{h,k}^\beta - L_h^{-\beta}) \Pi_h g, (Q_{h,k}^\beta + L_h^{-\beta}) \Pi_h g}{H} \bigr|
%        &\leq \norm{(L_h^{-\beta} - Q_{h,k}^\beta) \Pi_h g}{H}
%            ( \norm{L_h^{-\beta}}{\cL(V_h)} + \norm{Q_{h,k}^{\beta}}{\cL(V_h)} ) \norm{\Pi_h g}{H} \\
         \lesssim e^{-\pi^2/(2k)} \norm{g}{H}^2, \\
%\end{align*}
%%
%and, similarly,
%%
%\begin{align*}
    \text{(IIIb)}
        &= \bigl| \norm{(Q_{h,k}^\beta - L_h^{-\beta}) \white_h^\cE }{L_2(\Omega;H)}^2
            + 2 \, \scalar{(Q_{h,k}^\beta - L_h^{-\beta}) \white_h^\cE, L_h^{-\beta} \white_h^\cE}{L_2(\Omega; H)} \bigr| \\
        &\leq \sum_{j=1}^{N_h} \norm{(Q_{h,k}^\beta - L_h^{\beta}) e_{j,h} }{H}^2
            + 2 \, \Bigl| \sum_{j=1}^{N_h} \lambda_{j,h}^{-\beta} \scalar{(Q_{h,k}^\beta - L_h^{\beta}) e_{j,h}, e_{j,h} }{H} \Bigr| \\
        &\lesssim e^{-\pi^2/k} N_h + e^{-\pi^2/(2k)}  \sum_{j=1}^{N_h} \lambda_{j,h}^{-\beta}
%        &\lesssim e^{-\pi^2/k} N_h + e^{-\pi^2/(2k)} (1 + N_h^{1-\alpha\beta}) \\
         \lesssim e^{-\pi^2/k} h^{-d} + e^{-\pi^2/(2k)} (1 + f_{\alpha,\beta}(h) ),
\end{align*}
where we have used Assumption~\ref{ass:Vh} in the last estimate.
%
%\begin{align*}
%    \text{(III)} &\leq
%        \norm{u_h^\Phi - u_{h,k}^Q}{L_2(\Omega;H)} ( \norm{u_h^\Phi}{L_2(\Omega; H)} + \norm{u_{h,k}^Q}{L_2(\Omega;H)} ) \\
%        &\leq \norm{u_h^\Phi - u_{h,k}^Q}{L_2(\Omega;H)} ( \norm{L_h^{-\beta}}{\cL(V_h)} + \norm{Q_{h,k}^\beta}{\cL(V_h)} ) \norm{\white_h^\Phi + \Pi_h g}{L_2(\Omega;H)}.
%\end{align*}
%Since $\norm{L_h^{-\beta}}{\cL(V_h)} = \lambda_{1,h}^{-\beta}$,
%$\norm{Q_{h,k}^\beta}{\cL(V_h)} \lesssim \lambda_{1,h}^{-\beta}$
%for sufficiently small $k\in(0,k_0)$, and
%\begin{align*}
%    \norm{\white_h^\Phi + \Pi_h g}{L_2(\Omega;H)} \leq \norm{\white_h^\Phi}{L_2(\Omega;H)} + \norm{\Pi_h g}{H} \leq h^{-d/2} + \norm{g}{H},
%\end{align*}
%we conclude with~\eqref{e:err-quad} in Lemma~\ref{lem:err-quad} that
%\begin{align*}
%    \text{(III)}
%        &\lesssim \norm{u_h^\Phi - u_{h,k}^Q}{L_2(\Omega;H)} (h^{-d/2} + \norm{g}{H})
%         \lesssim e^{-\pi^2/(2k)} (h^{-d}
%         + \norm{g}{H}^2)
%\end{align*}
This proves~\eqref{e:err-weak}.
\end{proof}

\begin{table}[t]
  \centering
  \caption{\label{tab:rocs}Theoretical strong and weak type convergence rates}
  \begin{tabular}{lcc}
  \toprule
                                    & calibration & rate of convergence \\
  \cmidrule(r){2-3}
  strong error~\eqref{e:err-strong}  & \parbox[0pt][2em][c]{0cm}{} $k\leq -\tfrac{\pi^2}{2d\alpha\beta \ln(h)}$ & $\min\{d(\alpha\beta-1/2),r,s\}$ \\
  weak type error~\eqref{e:err-weak} & \parbox[0pt][4em][c]{0cm}{} $k\leq -\tfrac{\pi^2}{2 K_{\alpha,\beta}(h) }$
                                                                                            & $\begin{cases} \min\{d(2\alpha\beta-1),r\} & \text{if }g=0 \\
                                                                                               \min\{d(2\alpha\beta-1),r,s\} & \text{otherwise} \end{cases}$ \\
  \bottomrule
  \end{tabular}
%  \vspace{0.8\baselineskip}
%  \caption{Theoretical strong and weak type convergence rates}\label{tab:rocs}
\end{table}

\begin{remark}\label{rem:calibrate-hk}
The error estimates in~\eqref{e:err-strong} and~\eqref{e:err-weak}
imply that
the distance of the quadrature nodes $k$
has to be adjusted to
the finite element mesh size $h$.
Table~\ref{tab:rocs} shows
the calibration between $h$ and $k$
for error studies of strong and weak type
as well as the corresponding
theoretical convergence rates.
For the calibration, we set
\begin{align*}
    K_{\alpha,\beta}(h)
        := \begin{cases}
            d\alpha\beta \ln(h), &  \alpha\beta < 1, \\
            d \ln(h) - \max\{0, \ln(|\ln(h)|)\}, & \alpha\beta = 1, \\
            d(2\alpha\beta - 1) \ln(h), & \alpha\beta > 1.
            \end{cases}
\end{align*}
\end{remark}

\begin{remark}\label{rem:beta>1}
In the non-fractional case $\beta = 1$,
the discretized problem~\eqref{e:uhPhi}
is also non-fractional.
Therefore, realizations of its solution $u_h^\Phi = L_h^{-1}(\Pi_h g + \white_h^\Phi)$
can be computed directly and no quadrature is needed.
If $\beta \in (n,n+1]$ for some $n\in\bbN$,
one may apply the above described method and error estimates
to $\widetilde{L} := L^{n+1}$ and
$\widetilde{\beta} := \tfrac{\beta}{n+1} \in (0,1]$,
since $L^\beta = \widetilde{L}^{\widetilde{\beta}}$.
Yet, the finite element theory for the operator $\widetilde{L}$ may not be trivial.
\end{remark}

%=======================================================================================
\section{An application and numerical examples}\label{section:numexp}
%=======================================================================================

In the following numerical experiment
we take up the SPDE from~\eqref{e:statmodel}
in \S\ref{section:intro}.
More precisely,
with the objective of generating
computationally efficient approximations of
Gaussian Mat\'ern fields
on the unit cube $\cD := (0,1)^d$
in $d=1,2,3$ spatial dimensions,
we consider the following problem:
\begin{subequations}\label{e:numexp}
\begin{align}
\hspace*{2cm}
        (\kappa^2 - \Delta)^{\beta} u(\xvec) &= \white(\xvec), && \xvec \in \cD,  \label{e:numexp-a} \hspace*{2cm}\\
        u(\xvec)                             &= 0,             && \xvec \in \partial \cD, \label{e:numexp-b}
\end{align}
\end{subequations}
and study the above presented numerical method generating the approximation $u_{h,k}^Q$ in~\eqref{e:uQh}.
As already observed in Example~\ref{ex:eig:laplace}, the exponent of the eigenvalue growth
is given by $\alpha = 2/d$ in this case.

Furthermore, using a finite element discretization on uniform meshes
with continuous, piecewise linear basis functions,
Assumption~\ref{ass:Vh} is satisfied for this problem
in all three dimensions
with $r=s=q=2$, see Example~\ref{ex:ass:Vh}.
The condition in~\eqref{e:ass-alpha} of Theorem~\ref{thm:strong-conv}
becomes $\beta > d/4$. We emphasize that this
assumption is meaningful also from the statistical
point of view: on all of $\bbR^d$, $\beta > d/4$ corresponds to
a positive smoothness parameter
$\nu > 0$ of the Mat\'{e}rn covariance function~\eqref{e:materncov}.

Thus, if the quadrature step size $k$
and the finite element mesh width $h$ are calibrated
as indicated in Table~\ref{tab:rocs} in \S\ref{subsec:proof-remarks},
the theoretical rates of convergence
for $\beta \in (d/4,1)$ are $2\beta - d/2$ for the strong error
and $\min\{4\beta-d,2\}$ for the weak type error
according to Theorem~\ref{thm:strong-conv} and Corollary~\ref{cor:err-weak}.

\pagebreak

\begin{table}[t]
\centering
\caption{\label{tab:nodenumber}Numbers of finite element basis functions on the considered meshes for $d=1,2,3$ as well as
the corresponding numbers of quadrature nodes as a function of $\beta$ for the strong error study}
\begin{tabular}{lc|ccccc}
\toprule
  &     & \multicolumn{5}{c}{$\beta$}\\
  & $N_h$ & $3/8$  & $4/8$ & $5/8$ & $6/8$ & $7/8$\\
\cmidrule(r){2-7}
\multirow{4}{*}{$d=1$}  & 127   & 37  &  61  &  99  & 176  & 408 \\
                        & 255   & 48  &  77  & 129  & 229  & 533 \\
                        & 511   & 60  &  99  & 163  & 291  & 675 \\
                        & 1023  & 73  & 121  & 200  & 357  & 832 \\ \cmidrule(lr){1-7}
\multirow{4}{*}{$d=2$}  & 961   & -   & -    & 43   &  75  & 171 \\
                        & 3969  & -   & -    & 62   & 109  & 253 \\
                        & 16129 & -   & -    & 86   & 152  & 352 \\
                        & 65025 & -   & -    & 113  & 203  & 469 \\ \cmidrule(lr){1-7}
\multirow{3}{*}{$d=3$}  & 729   & -   & -    & -    & -    & 55 \\
                        & 6859  & -   & -    & -    & -    & 105 \\
                        & 59319 & -   & -    & -    & -    & 172 \\
\bottomrule
\end{tabular}
%\vspace{0.8\baselineskip}
%\caption{Numbers of finite element basis functions on the considered meshes for $d=1,2,3$ as well as
%the corresponding numbers of quadrature nodes as a function of $\beta$ for the strong error study}
\end{table}

For Problem~\eqref{e:numexp} with $\kappa=0.5$,
we investigate the empirical convergence rates
\begin{enumerate}
\item\label{enum:numexp:i} of the strong error for $d=1,2,3$ and $\beta = \tfrac{2d + n}{8}$, $n\in\{1,\ldots,7-2d\}$;
\item\label{enum:numexp:ii} of the weak type error for $d=1,2,3$
    and $\beta = \tfrac{2d + 1}{8}$.
\end{enumerate}
In each dimension, we use a finite element method
in space with continuous, piecewise linear basis functions
on uniform meshes with mesh
diameter $h$, mesh nodes
$\xvec_1,\ldots, \xvec_{N_h}$
and corresponding mass matrix $\Mmat$.
We choose $k = -1/(\beta\ln h)$. The numbers of finite element basis functions
and the corresponding numbers of quadrature nodes
depending on $\beta$ used in the strong error study are shown in Table~\ref{tab:nodenumber}.

For~\ref{enum:numexp:i}, measuring the strong mean-square error
between the exact solution $u$ to our model problem~\eqref{e:numexp}
and the approximation $u_{h,k}^Q$ in~\eqref{e:uQh},
we proceed as follows:
First, samples of an overkill approximation
of the white noise
\begin{align*}
    \whiteok := \sum_{\theta_1=1}^{\Nok} \ldots \sum_{\theta_d=1}^{\Nok}
            \xi_{(\theta_1,\ldots,\theta_d)} e_{(\theta_1,\ldots,\theta_d)}
\end{align*}
are generated and
evaluated on a uniform overkill mesh of $\bar{\cD}$ with $\Nok^d$ nodes.
Here, $\{\xi_{\thetavec}\}$ are independent standard normally distributed random variables
and $\{e_{\thetavec}\}$ are the eigenfunctions
in~\eqref{e:eigfct:laplace}.
The approximation $\whiteok$ corresponds to a truncated Karhunen--Lo\`{e}ve expansion
with $\Nok^d$ terms.
From this, samples of the overkill solution $u_{\mathrm{ok}}$
are obtained via
\begin{align*}
    u_{\mathrm{ok}} := (\kappa^2 - \Delta)^{-\beta} \whiteok =
            \sum_{\theta_1=1}^{\Nok} \ldots \sum_{\theta_d=1}^{\Nok}
                \lambda_{(\theta_1,\ldots,\theta_d)}^{-\beta} \xi_{(\theta_1,\ldots,\theta_d)} e_{(\theta_1,\ldots,\theta_d)},
\end{align*}
where $\{\lambda_{\thetavec}\}$ are the eigenvalues
from~\eqref{e:eigval:laplace}.
For the sake of generating comparable samples of the approximation
$u_{h,k}^Q$, we consider the same realizations of $\whiteok$
and use the load vector $\widetilde{\bvec}$ with entries
$\widetilde{b}_i = \scalar{\whiteok, \phi_{i,h}}{L_2(\cD)}$
instead of $\bvec \sim \cN(\mathbf{0},\Mmat)$ %with $b_i = \scalar{\white_h^\Phi, \phi_{i,h}}{L_2(\cD)}$
from Remark~\ref{rem:rhs-easysample},
as $(\scalar{\white,\phi_{i,h}}{L_2(\cD)})_{i=1}^{N_h} \sim \cN(\mathbf{0},\Mmat)$
and we treat $\whiteok$ as the true white noise.
The resulting approximation is denoted by $\widetilde{u}_{h,k}$.
We choose $\Nok = 2^{18} + 1$ for $d=1$,
$\Nok = 2^{12} + 1$ for $d=2$, and
$\Nok = 5 \cdot 2^{6} + 1$ for $d=3$.

\begin{figure}[t]
\begin{center}
\begin{minipage}[t]{0.476\linewidth}
\begin{center}
$d = 1$
\includegraphics[width=\linewidth]{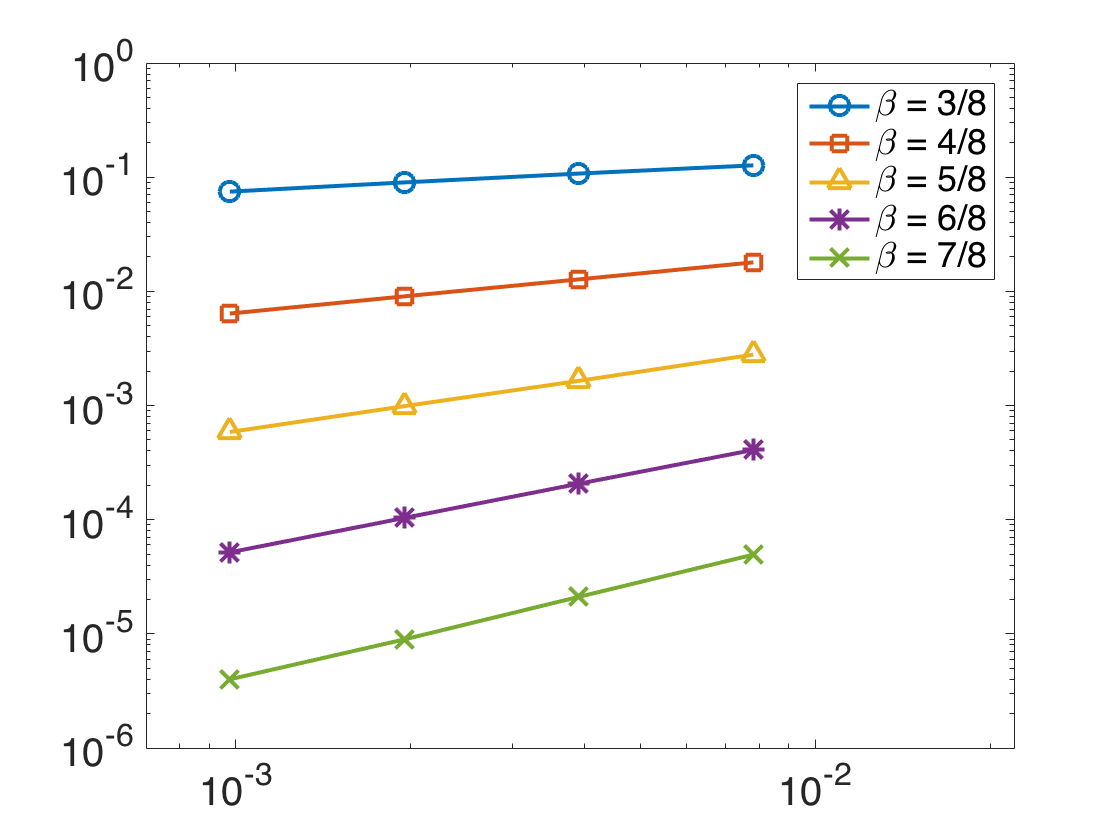}
\end{center}
\end{minipage}
\begin{minipage}[t]{0.476\linewidth}
\begin{center}
$d = 2$
\includegraphics[width=\linewidth]{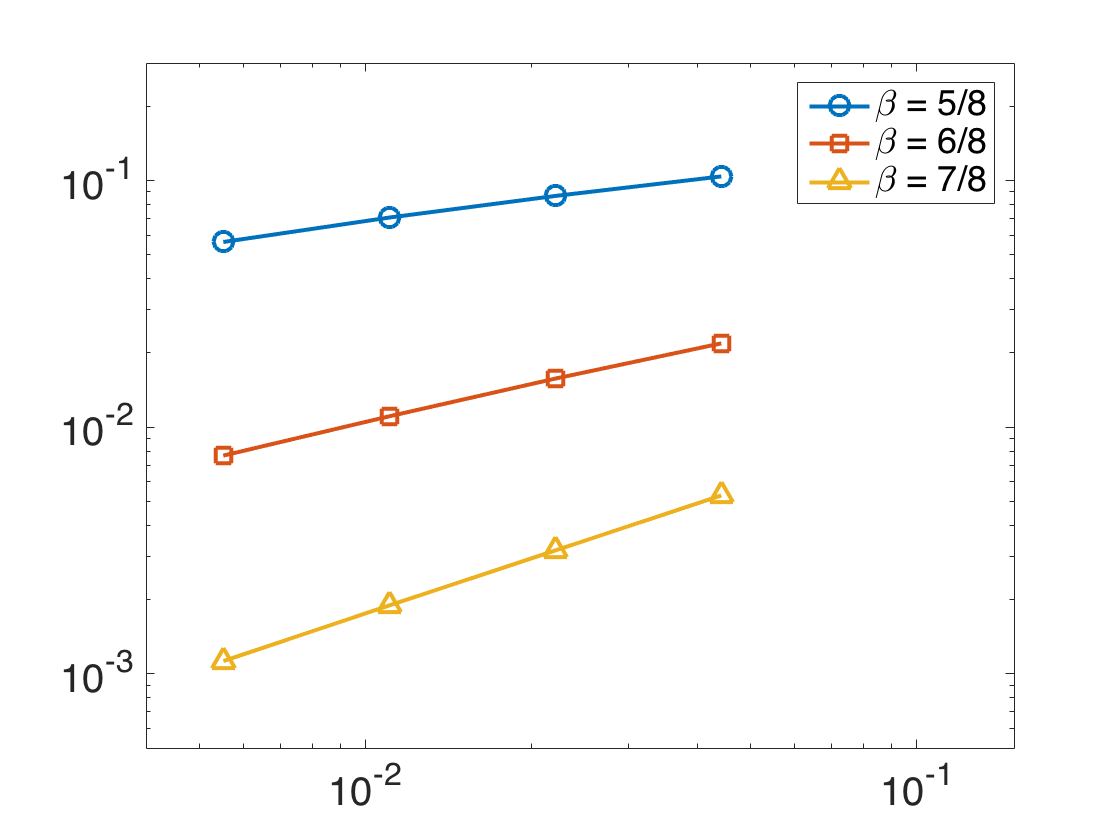}
\end{center}
\end{minipage}
\begin{minipage}[t]{0.476\linewidth}
\begin{center}
$\beta = 7/8$
\includegraphics[width=\linewidth]{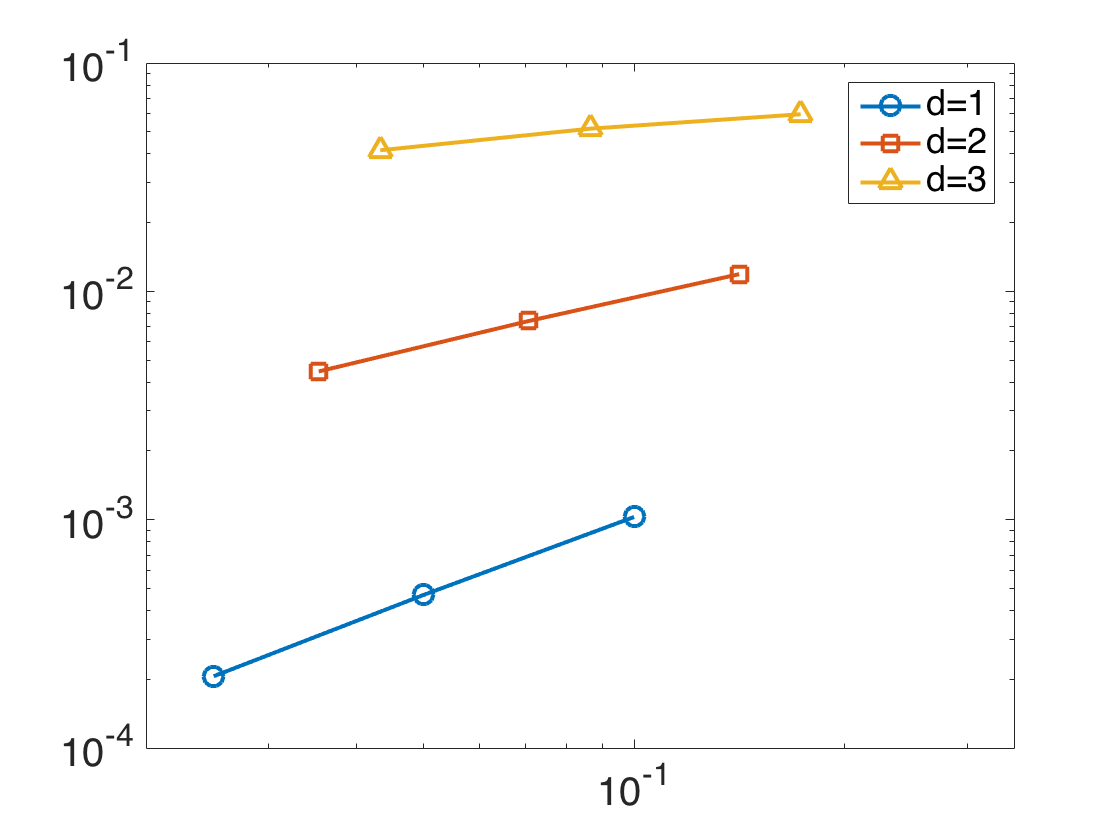}
\end{center}
\end{minipage}
\begin{minipage}[t]{0.476\linewidth}
\begin{center}
weak type error
\includegraphics[width=\linewidth]{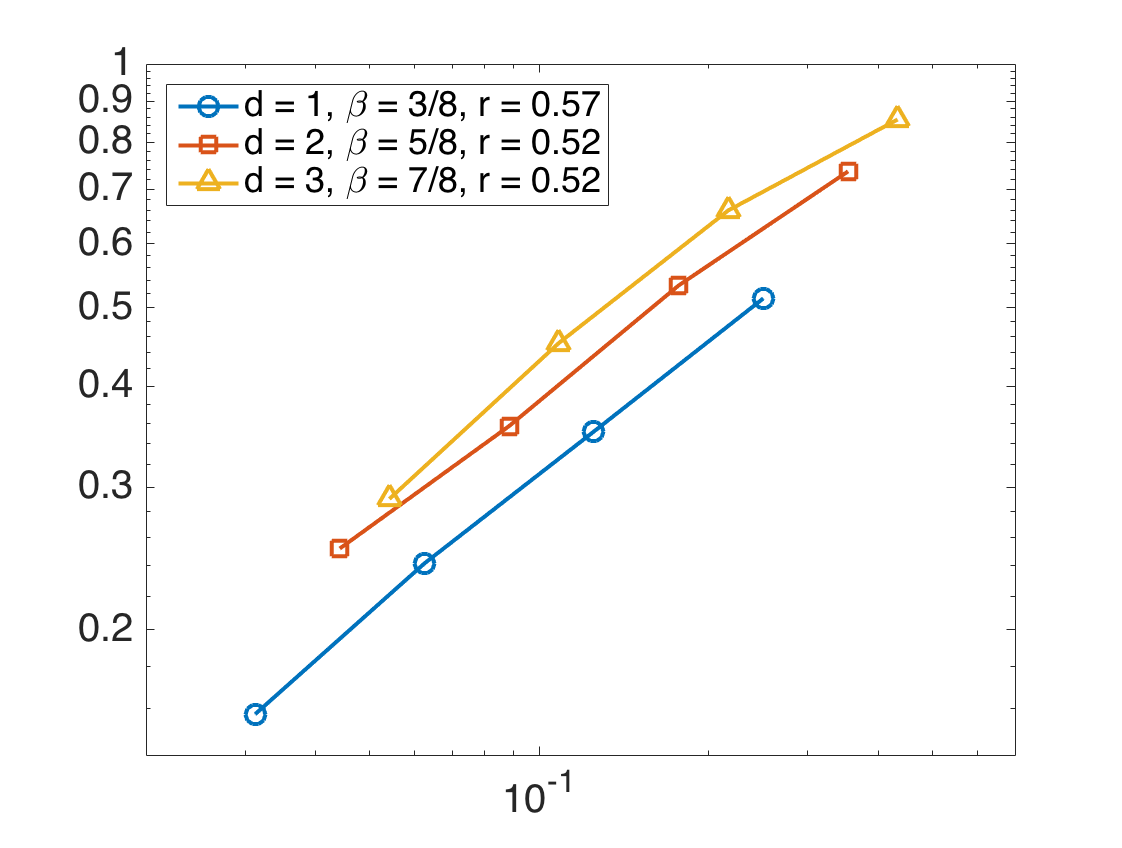}
\end{center}
\end{minipage}
\end{center}
\caption{The panels in the first row show the observed strong error %$\norm{u_{\mathrm{ok}} - u_{h,k}^Q}{L_2(\cD)}$
for different values of $\beta$ and $d=1,2$.
In the third graph the observed strong error for $\beta = 7/8$ and $d=1,2,3$ is displayed.
The corresponding observed strong convergence rates are shown in Table~\ref{tab:rates}.
The final graph shows the observed average weak type errors for three experiments in $d=1,2,3$ dimensions.
For all of them the theoretical rate of convergence is $0.5$ and the observed rates are shown in the legend.
All errors are shown as functions of the mesh size $h$ used in the computations in a log-log scale.}
\label{fig:errors}
\end{figure}

For each value of $\beta$ and in every spatial dimension,
we use $50$ samples of $\whiteok$ to generate samples
$u_{\mathrm{ok}}^{(1)}, \ldots, u_{\mathrm{ok}}^{(50)}$
of the overkill solution and of the numerical approximation
$\widetilde{u}_{h,k}^{(1)}, \ldots,\widetilde{u}_{h,k}^{(50)}$
for every mesh size $h$.
The observed strong errors are then computed as
the average $L_2$-errors
\begin{align*}
    \mathrm{err} := \tfrac{1}{50} \sum_{i=1}^{50} \sqrt{ \bigl(\mathbf{v}^{(i)}\bigr)^T
        \Mmat \, \mathbf{v}^{(i)}}, \qquad
    v^{(i)}_j := u_{\mathrm{ok}}^{(i)}(\xvec_j) - \widetilde{u}_{h,k}^{(i)}(\xvec_j).
\end{align*}
The results are shown in the first three panels of Figure~\ref{fig:errors}.
The data set $\{(h_\ell, \mathrm{err}_\ell)\}_\ell$
is then used to compute the observed rate of convergence
$\mathrm{r}$ as the least-squares solution to the linear regression
$\ln \mathrm{err} = \mathrm{c} + \mathrm{r} \ln h$
for each combination of $(d,\beta)$.
As shown in Table~\ref{tab:rates},
the resulting observed rates of convergence
are in accordance with the theoretical values predicted
by Theorem~\ref{thm:strong-conv}.

For \ref{enum:numexp:ii} we study the weak type error
$|\bbE[\norm{u}{L_2(\cD)}^2] - \bbE[\norm{u_{h,k}^Q}{L_2(\cD)}^2]|$
addressed in Corollary~\ref{cor:err-weak}.
Note that for this study no sample-wise comparison is needed
and, thus, the load vector is sampled from $\bvec \sim \cN(\mathbf{0},\Mmat)$
as discussed in Remark~\ref{rem:rhs-easysample}.
The variance of the exact solution can be computed directly
from the known eigenvalues of the differential operator as
$\bbE[\norm{u}{L_2(\cD)}^2] = \sum_{\thetavec} \lambda_{\thetavec}^{-2\beta}$.
The variance of $u_{h,k}^Q$ is approximated
via Monte Carlo integration by
\begin{align*}
    \bbE[\norm{u_{h,k}^Q}{L_2(\cD)}^2] \approx
        \tfrac{1}{N_{\mathrm{MC}}} \sum_{i=1}^{N_{\mathrm{MC}}}
        \bigl( \mathbf{u}^{(i)} \bigr)^{T} \Mmat \, \mathbf{u}^{(i)},
        \qquad
        u_j^{(i)} := u_{h,k}^{(i)}(\xvec_j),
\end{align*}
where the number of Monte Carlo samples is $N_{\mathrm{MC}} = 10^3$
and $u_{h,k}^{(i)}$ denotes a realization
of the numerical approximation $u_{h,k}^Q$.
The observed weak type errors and
the observed rates of convergence
are displayed in the fourth panel of Figure~\ref{fig:errors}.
The theoretical rate of convergence
predicted by Corollary~\ref{cor:err-weak}
is 0.5 for all three cases.

\begin{table}[t]
\centering
\caption{\label{tab:rates}Observed (resp.~theoretical) rates of convergence for the strong errors shown in Figure~\ref{fig:errors}}
\begin{tabular}{lccccc}
\toprule
&  \multicolumn{5}{c}{$\beta$}\\
&  $3/8$  & $4/8$ & $5/8$ & $6/8$ & $7/8$\\
\cmidrule(r){2-6}
$d=1$	&  0.25 (0.25) 	&  0.50 (0.5) 	&  0.75 (0.75) 	& 1.00 (1) 	& 1.21 (1.25)\\
$d=2$	& - 			& -    		&  0.29 (0.25)  	& 0.51 (0.5) 	&  0.74 (0.75)\\
$d=3$	& - 			& -    		& -  			& -  			&  0.26 (0.25)\\
\bottomrule
\end{tabular}%
%\vspace{0.8\baselineskip}
%\caption{Observed (resp.~theoretical) rates of convergence for the strong errors shown in Figure~\ref{fig:errors}}
\end{table}

%=======================================================================================
\section{Conclusion}\label{section:conclusion}
%=======================================================================================

We have considered the fractional order equation~\eqref{e:Lbeta}
with Gaussian white noise in a Hilbert space setting.
We have shown that the fractional operator $L^\beta$
is an isometric isomorphism %as a mapping
between the $\dot{H}$-spaces in~\eqref{e:Hs-embed}.
From this result and the mean-square regularity
of the white noise with respect to the $\dot{H}$-spaces,
we have deduced existence and uniqueness
of a solution $u$ to~\eqref{e:Lbeta}
with a certain regularity. %mean-square regularity.

We have proposed the approximation $u_{h,k}^Q$ in~\eqref{e:uQh}
based on two numerical ingredients:
\begin{enumerate*}
\item finite-dimensional subspaces $V_h$ of $\Hdot{1}$,
and \item the quadrature approximation~\eqref{e:Qhbeta}
of the inverse fractional operator.
\end{enumerate*}
The most advantageous and novel properties
of the corresponding numerical scheme
are
\begin{enumerate*}[label=(\alph*)]
\item that only solutions
    to integer order (i.e., local) elliptic equations have to be computed, and
\item that
    it does not require
    the knowledge of the eigenfunctions
    of the differential operator $L$.
\end{enumerate*}

Our main result, Theorem~\ref{thm:strong-conv},
shows strong mean-square
convergence of the approximation $u_{h,k}^Q$
to the exact solution $u$.
If the quadrature step size~$k$
and the finite element mesh width~$h$
are calibrated appropriately, see Table~\ref{tab:rocs},
the resulting strong convergence rate
depends only on the fractional order~$\beta$,
the dimension~$d$,
the eigenvalue growth~$\alpha$ of
the operator $L$, and
the approximation properties
of the finite element spaces.
In order to prove this result,
we have partitioned the strong error
in three terms: the truncation error,
the error caused by the finite
element discretization,
and the error of the quadrature approximation.
We have derived bounds for each
of these error terms separately
in \S\S\ref{subsec:err-trunc}--\ref{subsec:err-quad}.
By means of similar techniques,
we have proven a weak type error estimate in Corollary~\ref{cor:err-weak}.

Finally, in \S\ref{section:numexp}
we have applied the proposed numerical method
to an explicit problem with relevance for spatial statistics:
the solution $u$ to~\eqref{e:numexp}
can be regarded as an approximation
of a Gaussian Mat\'{e}rn field
on the unit cube $(0,1)^d$.
The performed numerical experiments
with continuous, piecewise linear
finite element basis functions
in $d=1,2,3$ spatial dimensions
verify the derived theoretical strong
and weak type convergence rates,
see Figure~\ref{fig:errors} and Table~\ref{tab:rates}.

We hope that these results and insights will prove
valuable for applications in spatial
statistics, which often require sampling from (approximations of)
Gaussian Mat\'{e}rn fields and their various extensions.

\bibliographystyle{IMANUM-BIB}
\bibliography{bkk_bib}

\end{document}

%% file: bkk_commands.tex
\usepackage[english]{babel}
\usepackage[utf8]{inputenc}

\usepackage{amsthm,amsfonts,amssymb,amsmath}
\usepackage{stmaryrd}
\numberwithin{equation}{section}
\usepackage{dsfont}
\usepackage{multirow}
\usepackage{url}
\usepackage{algorithmic,algorithm}
\usepackage{graphics}
\usepackage{booktabs}
\usepackage{epsfig}
\usepackage{epstopdf}
\usepackage{psfrag}               % Modifikation von Postscriptgrafiken
\usepackage{mathrsfs}
\usepackage{mathtools}            % Fuer \shortintertext
\usepackage{subfigure}
\usepackage{longtable}	%Tabellen mit Seitenumbruch --> evtl wieder entfernen

\usepackage{paralist}
\usepackage[inline]{enumitem}

\usepackage[pdftex,
            pdftitle={Numerics for fractional elliptic equations driven by spatial white noise},
            pdfauthor={David Bolin and K. Kirchner and M. Kovacs},
            bookmarksopen,
            colorlinks,
            linkcolor=black,
            urlcolor=black,
	    citecolor=black
]{hyperref}
\usepackage{hyperref}
\newcommand{\bbE}{\mathbb{E}}

\newcommand{\bbN}{\mathbb{N}}
\newcommand{\bbP}{\mathbb{P}}

\newcommand{\bbR}{\mathbb{R}}

\newcommand{\bbZ}{\mathbb{Z}}

\newcommand{\cA}{\mathcal{A}}

\newcommand{\cD}{\mathcal{D}}
\newcommand{\cE}{\mathcal{E}}

\newcommand{\cL}{\mathcal{L}}

\newcommand{\cN}{\mathcal{N}}
\newcommand{\cO}{\mathcal{O}}

\newcommand{\cW}{\mathcal{W}}

\newcommand{\scrD}{\mathscr{D}}

\newcommand{\gvec}{\mathbf{g}}
\newcommand{\bvec}{\mathbf{b}}
\newcommand{\xvec}{\mathbf{x}}
\newcommand{\zvec}{\mathbf{z}}
\newcommand{\xivec}{\boldsymbol{\xi}}
\newcommand{\thetavec}{\boldsymbol{\theta}}

\newcommand{\Mmat}{\mathbf{M}}
\newcommand{\Rmat}{\mathbf{R}}
\newcommand{\Imat}{\mathbf{I}}

\newcommand{\norm}[2]{     \| #1       \|_{ #2 }}

\newcommand{\normiii}[2]{\vert\kern-0.25ex\vert\kern-0.25ex\vert #1 \vert\kern-0.25ex \vert\kern-0.25ex\vert_{ #2 }}
\newcommand{\Normiii}[2]{\left\vert\kern-0.25ex\left\vert\kern-0.25ex\left\vert #1 \right\vert\kern-0.25ex\right\vert\kern-0.25ex\right\vert_{ #2 }}
\newcommand{\scalar}[2]{     ( #1       )_{ #2 }}

\newcommand{\duality}[2]{     \langle #1       \rangle_{ #2 }}

\newcommand{\Laplace}{\Delta}

\newcommand{\white}{\cW}

\newcommand{\tr}{\operatorname{tr}}
\newcommand{\rd}{\mathrm{d}}
\newcommand{\from}{\colon}
\newcommand{\Nok}{N_{\mathrm{ok}}}
\newcommand{\whiteok}{\white_{\mathrm{ok}}}

\newcommand{\Hdot}[1]{\dot{H}^{#1}}

\graphicspath{{Images/}}

\newtheorem{lemma}{Lemma}[section]
\newtheorem{proposition}[lemma]{Proposition}
\newtheorem{theorem}[lemma]{Theorem}
\newtheorem{corollary}[lemma]{Corollary}

\theoremstyle{remark}
\newtheorem{remark}[lemma]{Remark}

\theoremstyle{definition}

\newtheorem{assumption}[lemma]{Assumption}
\newtheorem{example}[lemma]{Example}

%% file: bkk-strong-ima-bib.bbl
\begin{thebibliography}{}

\bibitem[Babuska {\em et~al.}(2004)Babuska, Tempone, \& Zouraris]{babuska2004}
{\sc Babuska, I., Tempone, R. \& Zouraris, G.~E.} (2004)
\newblock Galerkin finite element approximations of stochastic elliptic partial
  differential equations.
\newblock {\em SIAM J.\ Numer.\ Anal.}, {\bf 42}, 800--825.

\bibitem[Baeumer {\em et~al.}(2015)Baeumer, Kov\'acs, \&
  Sankaranarayanan]{baeumer2015}
{\sc Baeumer, B., Kov\'acs, M. \& Sankaranarayanan, H.} (2015)
\newblock Higher order {G}r\"unwald approximations of fractional derivatives
  and fractional powers of operators.
\newblock {\em Trans.\ Amer.\ Math.\ Soc.}, {\bf 367}, 813--834.

\bibitem[Balakrishnan(1960)Balakrishnan]{balakrishnan1960}
{\sc Balakrishnan, A.~V.} (1960)
\newblock Fractional powers of closed operators and the semigroups generated by
  them.
\newblock {\em Pacific J.\ Math.}, {\bf 10}, 419--437.

\bibitem[Banerjee {\em et~al.}(2008)Banerjee, Gelfand, Finley, \&
  Sang]{banerjee2008gaussian}
{\sc Banerjee, S., Gelfand, A.~E., Finley, A.~O. \& Sang, H.} (2008)
\newblock Gaussian predictive process models for large spatial data sets.
\newblock {\em J.\ R.\ Stat.\ Soc.\ Series B Stat.\ Methodol.}, {\bf 70},
  825--848.

\bibitem[Bhatt {\em et~al.}(2015)Bhatt, Weiss, Cameron, Bisanzio, Mappin,
  Dalrymple, Battle, Moyes, Henry, Eckhoff, {\em et~al.}]{bhatt2015effect}
{\sc Bhatt, S., Weiss, D., Cameron, E., Bisanzio, D., Mappin, B., Dalrymple,
  U., Battle, K., Moyes, C., Henry, A., Eckhoff, P. {\em et~al.}} (2015)
\newblock The effect of malaria control on {P}lasmodium falciparum in {A}frica
  between 2000 and 2015.
\newblock {\em Nature\/}, {\bf 526}, 207--211.

\bibitem[Bolin(2014)Bolin]{bolin14}
{\sc Bolin, D.} (2014)
\newblock Spatial {M}at\'ern fields driven by non-{G}aussian noise.
\newblock {\em Scand.\ J.\ Stat.}, {\bf 41}, 557--579.

\bibitem[Bolin \& Lindgren(2011)Bolin \& Lindgren]{bolin11}
{\sc Bolin, D. \& Lindgren, F.} (2011)
\newblock Spatial models generated by nested stochastic partial differential
  equations, with an application to global ozone mapping.
\newblock {\em Ann.\ Appl.\ Stat.}, {\bf 5}, 523--550.

\bibitem[Bonito {\em et~al.}(2017)Bonito, Lei, \& Pasciak]{bonito2017}
{\sc Bonito, A., Lei, W. \& Pasciak, J.~E.} (2017)
\newblock The approximation of parabolic equations involving fractional powers
  of elliptic operators.
\newblock {\em J.\ Comp.\ Appl.\ Math.}, {\bf 315}, 32--48.

\bibitem[Bonito \& Pasciak(2015)Bonito \& Pasciak]{bonito2015}
{\sc Bonito, A. \& Pasciak, J.~E.} (2015)
\newblock Numerical approximation of fractional powers of elliptic operators.
\newblock {\em Math.\ Comp.}, {\bf 84}, 2083--2110.

\bibitem[Caffarelli \& Silvestre(2007)Caffarelli \& Silvestre]{caffarelli2007}
{\sc Caffarelli, L. \& Silvestre, L.} (2007)
\newblock An extension problem related to the fractional {L}aplacian.
\newblock {\em Comm.\ Partial Differential Equations\/}, {\bf 32}, 1245--1260.

\bibitem[Cao {\em et~al.}(2007)Cao, Yang, \& Yin]{cao2007}
{\sc Cao, Y., Yang, H. \& Yin, L.} (2007)
\newblock Finite element methods for semilinear elliptic stochastic partial
  differential equations.
\newblock {\em Numerische Mathematik\/}, {\bf 106}, 181--198.

\bibitem[Courant \& Hilbert(1962)Courant \& Hilbert]{courant1962}
{\sc Courant, R. \& Hilbert, D.} (1962)
\newblock {\em Methods of Mathematical Physics: Vol. I\/}.
\newblock Wiley Classics Library.
\newblock Interscience Publishers.

\bibitem[D{\"o}lz {\em et~al.}(2017)D{\"o}lz, Harbrecht, \& Schwab]{Dolz2017}
{\sc D{\"o}lz, J., Harbrecht, H. \& Schwab, C.} (2017)
\newblock Covariance regularity and $\mathcal{H}$-matrix approximation for
  rough random fields.
\newblock {\em Numerische Mathematik\/}, {\bf 135}, 1045--1071.

\bibitem[Du \& Zhang(2003)Du \& Zhang]{du2003}
{\sc Du, Q. \& Zhang, T.} (2003)
\newblock Numerical approximation of some linear stochastic partial
  differential equations driven by special additive noises.
\newblock {\em SIAM J.\ Numer.\ Anal.}, {\bf 40}, 1421--1445.

\bibitem[Fuglstad {\em et~al.}(2015)Fuglstad, Lindgren, Simpson, \&
  Rue]{fuglstad2015}
{\sc Fuglstad, G.-A., Lindgren, F., Simpson, D. \& Rue, H.} (2015)
\newblock Exploring a new class of non-stationary spatial {G}aussian random
  fields with varying local anisotropy.
\newblock {\em Stat.\ Sin.}, 115--133.

\bibitem[Furrer {\em et~al.}(2006)Furrer, Genton, \&
  Nychka]{furrer2006covariance}
{\sc Furrer, R., Genton, M.~G. \& Nychka, D.} (2006)
\newblock Covariance tapering for interpolation of large spatial datasets.
\newblock {\em J.\ Comput.\ Graph.\ Stat.}, {\bf 15}, 502--523.

\bibitem[Gavrilyuk(1996)Gavrilyuk]{gavrilyuk1996}
{\sc Gavrilyuk, I.~P.} (1996)
\newblock An algorithmic representation of fractional powers of positive
  operators.
\newblock {\em Numer.\ Funct.\ Anal.\ Optim.}, {\bf 17}, 293--305.

\bibitem[Gavrilyuk {\em et~al.}(2004)Gavrilyuk, Hackbusch, \&
  Khoromskij]{gavrilyuk2004}
{\sc Gavrilyuk, I.~P., Hackbusch, W. \& Khoromskij, B.~N.} (2004)
\newblock Data-sparse approximation to the operator-valued functions of
  elliptic operator.
\newblock {\em Math.\ Comp.}, {\bf 73}, 1297--1324.

\bibitem[Gavrilyuk {\em et~al.}(2005)Gavrilyuk, Hackbusch, \&
  Khoromskij]{gavrilyuk2005}
{\sc Gavrilyuk, I.~P., Hackbusch, W. \& Khoromskij, B.~N.} (2005)
\newblock Hierarchical tensor-product approximation to the inverse and related
  operators for high-dimensional elliptic problems.
\newblock {\em Computing\/}, {\bf 74}, 131--157.

\bibitem[Gy\"{o}ngy \& Mart\'{i}nez(2006)Gy\"{o}ngy \&
  Mart\'{i}nez]{gyongy2006}
{\sc Gy\"{o}ngy, I. \& Mart\'{i}nez, T.} (2006)
\newblock On numerical solution of stochastic partial differential equations of
  elliptic type.
\newblock {\em Stochastics\/}, {\bf 78}, 213--231.

\bibitem[Jin {\em et~al.}(2015)Jin, Lazarov, Pasciak, \& Rundell]{jin2015}
{\sc Jin, B., Lazarov, R., Pasciak, J. \& Rundell, W.} (2015)
\newblock Variational formulation of problems involving fractional order
  differential operators.
\newblock {\em Math.\ Comp.}, {\bf 84}, 2665--2700.

\bibitem[Kov\'acs {\em et~al.}(2011)Kov\'acs, Lindgren, \& Larsson]{kovacs2011}
{\sc Kov\'acs, M., Lindgren, F. \& Larsson, S.} (2011)
\newblock Spatial approximation of stochastic convolutions.
\newblock {\em J. Comput. Appl. Math.}, {\bf 235}, 3554--3570.

\bibitem[Lai {\em et~al.}(2015)Lai, Biedermann, Ekpo, Garba, Mathieu, Midzi,
  Mwinzi, N'Goran, Raso, Assar{\'e}, {\em et~al.}]{lai2015spatial}
{\sc Lai, Y.-S., Biedermann, P., Ekpo, U.~F., Garba, A., Mathieu, E., Midzi,
  N., Mwinzi, P., N'Goran, E.~K., Raso, G., Assar{\'e}, R.~K. {\em et~al.}}
  (2015)
\newblock Spatial distribution of schistosomiasis and treatment needs in
  sub-{S}aharan {A}frica: a systematic review and geostatistical analysis.
\newblock {\em Lancet Infect.\ Dis.}, {\bf 15}, 927--940.

\bibitem[Lindgren {\em et~al.}(2011)Lindgren, Rue, \&
  Lindstr\"{o}m]{lindgren11}
{\sc Lindgren, F., Rue, H. \& Lindstr\"{o}m, J.} (2011)
\newblock An explicit link between {G}aussian fields and {G}aussian {M}arkov
  random fields: the stochastic partial differential equation approach (with
  discussion).
\newblock {\em J.\ R.\ Stat.\ Soc.\ Series B Stat.\ Methodol.}, {\bf 73},
  423--498.

\bibitem[Nochetto {\em et~al.}(2015)Nochetto, Ot{\'a}rola, \&
  Salgado]{nochetto2015}
{\sc Nochetto, R.~H., Ot{\'a}rola, E. \& Salgado, A.~J.} (2015)
\newblock A {PDE} approach to fractional diffusion in general domains: A priori
  error analysis.
\newblock {\em Found.\ Comp.\ Math.}, {\bf 15}, 733--791.

\bibitem[Nychka {\em et~al.}(2015)Nychka, Bandyopadhyay, Hammerling, Lindgren,
  \& Sain]{nychka2015multiresolution}
{\sc Nychka, D., Bandyopadhyay, S., Hammerling, D., Lindgren, F. \& Sain, S.}
  (2015)
\newblock A multiresolution {G}aussian process model for the analysis of large
  spatial datasets.
\newblock {\em J.\ Comput.\ Graph.\ Stat.}, {\bf 24}, 579--599.

\bibitem[Roop(2006)Roop]{roop2006}
{\sc Roop, J.~P.} (2006)
\newblock Computational aspects of {FEM} approximation of fractional advection
  dispersion equations on bounded domains in {${\Bbb R}^2$}.
\newblock {\em J.\ Comput.\ Appl.\ Math.}, {\bf 193}, 243--268.

\bibitem[Simpson(2008)Simpson]{simpson08phd}
{\sc Simpson, D.} (2008)
\newblock Krylov subspace methods for approximating functions of symmetric
  positive definite matrices with applications to applied statistics and
  anomalous diffusion.
\newblock {\em PhD thesis, Queensland University of Technology\/}.

\bibitem[Simpson {\em et~al.}(2012)Simpson, Lindgren, \& Rue]{simpson2012think}
{\sc Simpson, D., Lindgren, F. \& Rue, H.} (2012)
\newblock Think continuous: {M}arkovian {G}aussian models in spatial
  statistics.
\newblock {\em Spat.\ Stat.}, {\bf 1}, 16--29.

\bibitem[Stein(1999)Stein]{stein1999}
{\sc Stein, M.~L.} (1999)
\newblock {\em Interpolation of Spatial Data: Some Theory for Kriging\/}.
\newblock Springer Series in Statistics.
\newblock Springer New York.

\bibitem[Strang \& Fix(2008)Strang \& Fix]{strang2008}
{\sc Strang, G. \& Fix, G.} (2008)
\newblock {\em An Analysis of the Finite Element Method\/}.
\newblock Wellesley-Cambridge Press.

\bibitem[Sun {\em et~al.}(2012)Sun, Li, \& Genton]{sun2012geostatistics}
{\sc Sun, Y., Li, B. \& Genton, M.~G.} (2012)
\newblock Geostatistics for large datasets.
\newblock {\em Advances and challenges in space-time modelling of natural
  events\/}.
\newblock Springer, pp. 55--77.

\bibitem[Thomee(2007)Thomee]{thomee2007}
{\sc Thomee, V.} (2007)
\newblock {\em Galerkin Finite Element Methods for Parabolic Problems\/}.
\newblock Springer Series in Computational Mathematics.
\newblock Springer Berlin Heidelberg.

\bibitem[Wallin \& Bolin(2015)Wallin \& Bolin]{wallin15}
{\sc Wallin, J. \& Bolin, D.} (2015)
\newblock Geostatistical modelling using non-{G}aussian {M}at\'ern fields.
\newblock {\em Scand.\ J.\ Stat.}, {\bf 42}, 872--890.

\bibitem[Whittle(1954)Whittle]{whittle54}
{\sc Whittle, P.} (1954)
\newblock On stationary processes in the plane.
\newblock {\em Biometrika\/}, {\bf 41}, 434--449.

\bibitem[Whittle(1963)Whittle]{whittle63}
{\sc Whittle, P.} (1963)
\newblock Stochastic processes in several dimensions.
\newblock {\em Bull.\ Internat.\ Statist.\ Inst.}, {\bf 40}, 974--994.

\bibitem[Yosida(1995)Yosida]{yosida1995}
{\sc Yosida, K.} (1995)
\newblock {\em Functional Analysis\/}.
\newblock Classics in Mathematics.
\newblock Springer Berlin Heidelberg.

\bibitem[Zhang {\em et~al.}(2016)Zhang, Rozovskii, \& Karniadakis]{zhang2016}
{\sc Zhang, Z., Rozovskii, B. \& Karniadakis, G.~E.} (2016)
\newblock Strong and weak convergence order of finite element methods for
  stochastic {PDE}s with spatial white noise.
\newblock {\em Numerische Mathematik\/}, {\bf 134}, 61--89.

\end{thebibliography}
